\documentclass[11pt]{amsart}

\usepackage[]{graphicx}
\usepackage[abs]{overpic}
\usepackage{amssymb}
\usepackage{amsmath}
\usepackage{amsthm}
\usepackage{mathrsfs}
\usepackage{comment}
\usepackage{layout}

\theoremstyle{plain}
\newtheorem{thm}{Theorem}[section]
\newtheorem{lem}[thm]{Lemma}

\newtheorem{con}[thm]{Conjecture}

\theoremstyle{example}

\newtheorem{prob}[thm]{Problem}

\theoremstyle{definition}
\newtheorem{defn}[thm]{Definition}

\newtheorem{rem}[thm]{Remark}

\graphicspath{{./figures/}, {../figures/}, {/arXiv}}

\title[Yokota type invariants from ]{Yokota type invariants derived from non-integral highest weight representations of $\mathcal{U}_q(sl_2)$}
\author{Atsuhiko Mizusawa}
\author{Jun Murakami}
\address{Department of Mathematics, Faculty of Science and Engineering, Waseda University, 3-4-1 Okubo, Shinjuku-ku, Tokyo 169-8555, Japan}
\email[Jun Murakami]{murakami@waseda.jp}
\email[Atsuhiko Mizusawa]{a\_mizusawa@aoni.waseda.jp}
\thanks{The first author was partially supported by a Waseda University Grant for Special Research Projects (Project number: 2014A-345) and JSPS KAKENHI Grant Number 25287014.}
\thanks{The second author was partially supported by JSPS KAKENHI Grant Number 25610022.}

\begin{document}

\begin{abstract}
We define invariants for colored oriented spatial graphs by generalizing CM invariants \cite{CM}, which were defined via non-integral highest weight representations of $\mathcal{U}_q(sl_2)$. We apply the same method to define Yokota's invariants \cite{Yo}, and we call these invariants Yokota type invariants. Then we propose a volume conjecture of the Yokota type invariants of plane graphs which relates to volumes of hyperbolic polyhedra corresponding to the graphs, and check it numerically for some square pyramids and pentagonal pyramids.
\end{abstract}


\maketitle

\section{introduction} \label{sec1}
\par
In this paper, we define quantum invariants for colored oriented spatial graphs and propose a volume conjecture of these invariants of plane graphs which relates to volumes of hyperbolic convex polyhedra corresponding to the graphs. The part of this paper was announced in \cite{MM}. 
\par
In low dimensional topology, the volume conjecture is an open problem which gives a relation between quantum invariants of knots, which are defined combinatorially from representations of the quantum group $\mathcal{U}_q(sl_2)$, and geometric invariants like the hyperbolic volumes of their complements. The original volume conjecture was proposed by R.~Kashaev \cite{Ka} using the Kashaev invariants, then reformulated by H.~Murakami and the second author using the colored Jones polynomial \cite{MuMu}. Let $J_{n, K} (q)$ be the $n$-th colored Jones polynomial of a knot $K$ which is defined from an $n$-dimensional representation of $\mathcal{U}_q(sl_2)$. The conjecture say that, for a hyperbolic knot $K$ in $S^3$, the hyperbolic volume of its complement appears in the growth rate of the value of the $n$-th colored Jones polynomial $J_{n, k}(\exp(2\pi i/n))$: 
$$\lim_{n\rightarrow \infty} \frac{2\pi}{n}  \log (\,|J_{n, k}(\exp(2\pi i/n))|\,)={\rm Vol}(S^3\setminus k),$$
$$({\rm i.e. } \,\,|J_{n, k}(\exp(2\pi i/n))|\sim\exp\left(\frac{n{\rm Vol}(S^3\setminus k)}{2\pi}\right) {\rm as}\,\,n\rightarrow \infty.)$$
Many versions of the volume conjecture were proposed \cite{MMOTY}, \cite{HMu}, \cite{CGB}, \cite{Mu3}. 
\par
For the quantum invariants of trivalent spatial graphs, some volume conjectures were proposed and partially proved. F.~Costantino and the second author defined invariants for framed oriented colored trivalent graphs through non-integral highest weight representations of $\mathcal{U}_q(sl_2)$ \cite{CM}. We refer to these invariants as \textit{CM invariants} in this paper. It was proved that the CM invariants have the volume conjecture type relation for tetrahedron graphs. The restriction of the CM invariant to links is the colored Alexander invariant \cite{ADO}, \cite{Mu2}. For this invariant, a volume conjecture was proposed which relates to the volumes of cone manifolds whose cone singular sets are links \cite{Mu}, \cite{ChMu}.
The volume conjecture for augmented graphs was proved \cite{vdV} and the volume conjecture of the Kauffman brackets of plane trivalent graphs (or the Kirillov-Reshetikhin invariants \cite{KR}) was proposed \cite{CGvdV}, which relates to volumes of the hyperbolic polyhedra whose one-skeletons are the plane graphs.
\par
In this paper, the CM invariants are generalized for multivalent graphs and their volume conjecture is extended for general hyperbolic convex polyhedra.
We define invariants of colored oriented spatial graphs whose valencies are more than or equal to three. We call these invariants \textit{Yokota type invariants} since they are constructed by the same method to define Yokota's invariants \cite{Yo}. Y.~Yokota applied his method to the Kirillov-Reshetikhin invariants for trivalent graphs, and here we apply it to the CM invariants. 
 As a natural extension, the Yokota type invariants should have the same property as  the volume conjecture which gives a relation between the Yokota type invariants of a plane graph and the volume of a hyperbolic convex polyhedron whose one-skeleton coincides with the plane graph. We show numerical calculations to check this property for some hyperbolic square pyramids and pentagonal pyramids.
\par
The present paper is organized as follows. We review the CM invariants and their properties in Section \ref{sec2}. In Section \ref{sec3}, we define the Yokota type invariants and propose a volume conjecture for them. In Section \ref{sec4}, we show numerical calculations of the Yokota type invariants and observe their relations to the hyperbolic volumes.

\section{CM invariants} \label{sec2}
\par
In this section, we review CM invariants $\left<\,\cdot\,\right>_{\rm CM}$ 
for framed oriented colored trivalent graphs (see \cite{CM} for details). These invariants are defined through $n$-dimensional non-integral highest weight representations of $\mathcal{U}_{\xi_n}(sl_2)$, where $n\in\mathbb{N}$ and $\xi_n$ is the $2n$-th primitive root of unity $\exp(\pi \sqrt{-1}/n)$. Here, trivalent graphs may have circle components. For a non-half-integer complex number $a\in \mathbb{C}\setminus \frac{1}{2}\mathbb{Z}$, an $n$-dimensional representation $\rho_a: \mathcal{U}_{\xi_n}(sl_2) \rightarrow {\rm End}(V^a)$ is determine, where $\rho_a$ has the highest weight $a$ and $V^a$ is an $n$-dimensional vector space.  
\par
Let $\Gamma$ be a framed oriented trivalent graph. In a diagram, the framing is given by the blackboard framing. A \textit{color} is a non-half-integer complex number which represents a highest weight of the representation of $\mathcal{U}_{\xi_n}(sl_2)$. We attach colors $a, b, \dots$ to edges of $\Gamma$ (Figure \ref{fig01} left) and make a map from the edges of $\Gamma$ to the representation spaces $V^a, V^b, \dots$. The isomorphism between a dual space $(V^a)^*$ and $V^{n-1-a}$ induces the correspondence between an $a$ colored oriented edge and an $n-1-a$ colored opposite direction edge and  we identify such two edges (Figure \ref{fig01} right). We put $\overline{a}=n-1-a$.

\begin{figure}[ht] 
\[
\raisebox{-26 pt}{\begin{overpic}[bb=0 0 124 105, width=70 pt]{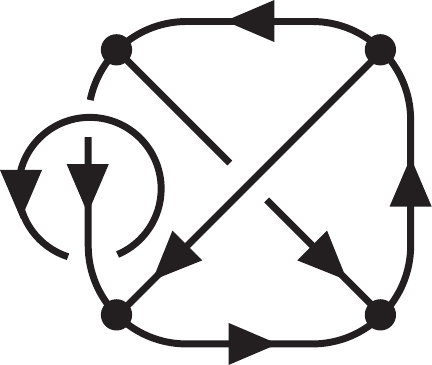}\put(56,26){$a$}\put(38,44){$b$}\put(31,11){$c$}\put(18,26){$d$}\put(45,11){$e$}\put(38,-7){$f$}\put(-6,27){$g$}\end{overpic}} 
\hspace{1.5cm}
\raisebox{-25 pt}{\includegraphics[bb=0 0 60 141, height=55 pt]{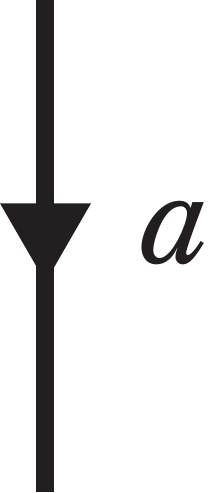}} \hspace{0.2cm}
\mbox{\LARGE{=}}
\hspace{0.2cm}\raisebox{-25 pt}{\includegraphics[bb=0 0 65 141, height=55 pt]{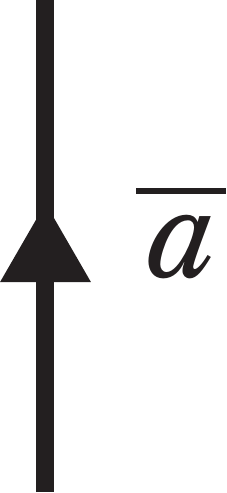}}
\]
\caption{A framed oriented colored trivalent graph (left) and the identification of colored edges (right).} \label{fig01}
\end{figure}
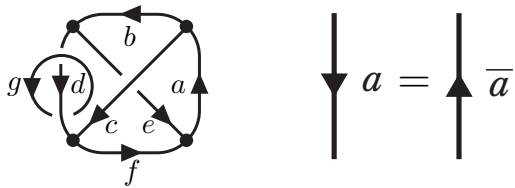

\par
For a trivalent vertex of $\Gamma$, if the colors $a, b, c$ of three edges around the vertex satisfy the following condition, we have a representation for the vertex canonically: 
\[a+b+c \in \{n-1, n, \dots , 2n-2\}, \hspace{1.0cm} \raisebox{-20 pt}{\begin{overpic}[bb=0 0 248 213, height=50 pt]{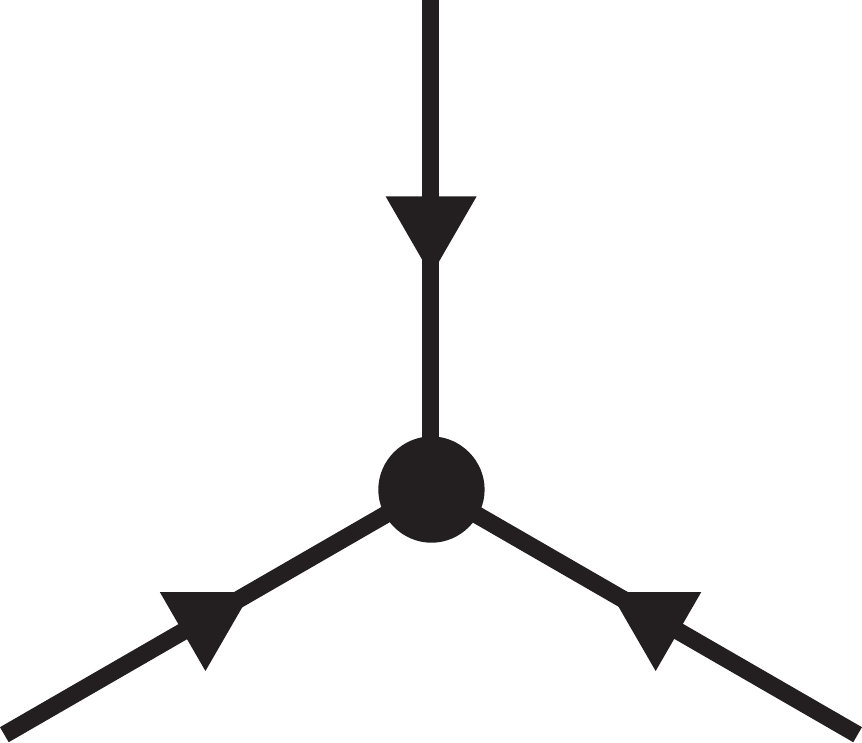}
\put(19, 34){$a$}
\put(46, 13){$b$}
\put(7, 13){$c$}
\end{overpic}} \]
where the orientations of the three edges are all toward the vertex. We say that three colors of edges around a vertex are \textit{admissible} if they satisfy the above condition. 
In the following, we assume three colors around every trivalent vertex are admissible and variable colors in summations run over all admissible colors, unless otherwise noted.
\par
In \cite{CM}, CM invariants of $\Gamma$ are defined through representations of $\mathcal{U}_{\xi_n}(sl_2)$ obtained from a $(1,1)$-tangle $T$ which is made by cutting any one edge of $\Gamma$. CM invariants are also calculated axiomatically by some local relations of diagrams. We shall see the relations. We prepare notations: 
$$\{a\}=\xi_n^a-\xi_n^{-a} \,\,\,(a\in \mathbb{C}), \,\,\,\,\,\,[a]=\frac{\{a\}}{\{1\}},\,\,\,\,\,\, \{k\}!= \prod_{j=1}^k\{j\}\,\,\, (k\in \mathbb{N}),$$
$$\left[ \begin{array}{c} a\\ b \end{array} \right] \! = \! \prod_{j=0}^{a-b-1} \! \frac{\{a-j\}}{\{a-b-j\}} \,\,\,(a,b \in \mathbb{C} \mbox{ s.t. }a-b\in\{0,1, \dots, n-1\}).$$
The following identities hold:
\begin{equation} \label{eq00}
\{a\} = \{n-a\} \,\,\,\,(a\in \mathbb{C})\,, 
\hspace{0.6cm} 
\{k\}!\, \{n-1-k\}!=\{n-1\}!\,\,\,\,(k\in\{0,1,\dots, n-1\}), 
\end{equation}
\begin{equation} \label{eq0}
\left[ \begin{array}{c} a\\ b \end{array} \right]=\left[ \begin{array}{c} n-1-b\\ n-1-a \end{array} \right], \hspace{0.5cm} 
\left[ \begin{array}{c} a\\ b \end{array} \right]=(-1)^{a-b}\left[ \begin{array}{c} a-n\\ b-n \end{array} \right].
\end{equation}
\par
The \textit{$6j$-symbols} $\{\,\cdot\,\}$ are the values determined by six colors. They are defined as coefficients of the relation (\ref{eq3}) below. For $a, b, c, d, e, f \in \mathbb{C} \setminus \frac{1}{2} \mathbb{Z}$ such that $a+b-c, a+f-e, b+d-f, d+c-e$ $ \in \mathbb{Z}$, the 6$j$-symbols are calculated by the following formula: 
\begin{eqnarray*}
& &\hspace{-0.5cm}\left\{ \begin{array}{ccc}  a & b & c \\ d & e & f \\ \end{array} \right\}=(-1)^{n-1+B_{afe}} \left[ \begin{array}{c} 2f+n\\ 2f+1 \end{array} \right]^{-1} \frac{\{B_{cde}\}!\{B_{abc}\}!}{\{B_{bdf}\}!\{B_{afe}\}!}\left[ \begin{array}{c} 2c\\ A_{abc}+1-n \end{array} \right]\left[ \begin{array}{c} 2c\\ B_{ced} \end{array} \right]^{-1}\\
& &\hspace{1.7cm}\times \sum_{z=s}^{S}(-1)^z\left[ \begin{array}{c} A_{afe}+1\\ 2e+z+1 \end{array} \right]\left[ \begin{array}{c} B_{aef}+z\\ B_{aef} \end{array} \right]\left[ \begin{array}{c} B_{bfd}+B_{cde}-z\\ B_{bfd} \end{array} \right]\left[ \begin{array}{c} B_{dec}+z\\ B_{dfb} \end{array} \right],
\end{eqnarray*}
where $s=\max(0, -B_{bdf}+B_{cde})$, $S=\min(B_{cde}, B_{afe})$, $A_{xyz}=x+y+z$ and $B_{xyz}=x+y-z$. The $6j$-symbols satisfy the orthogonal relation:
\begin{equation} \label{orth}
\sum_{f}\left\{ \begin{array}{ccc}  a & b & c \\ d & e & f \\ \end{array} \right\}
\left\{ \begin{array}{ccc}  d & b & f \\ a & e & g \\ \end{array} \right\} = \delta_{cg},
\end{equation}
where $f$ ranges over all the complex numbers such that both $b+d-f$ and $f+a-e$ are in $\{0,1, \dots, n-1\}$. Let $\{\, \cdot\,\}_{tet}$ be the value of the CM invariant of a tetrahedron graph, then they are described by using the $6j$-symbols:
\[\left\{ \begin{array}{ccc}  a & b & c \\ d & e & f \\ \end{array} \right\}_{tet}=
\left< \raisebox{-22 pt}{\begin{overpic}[bb=0 0 138 117, height=49 pt]{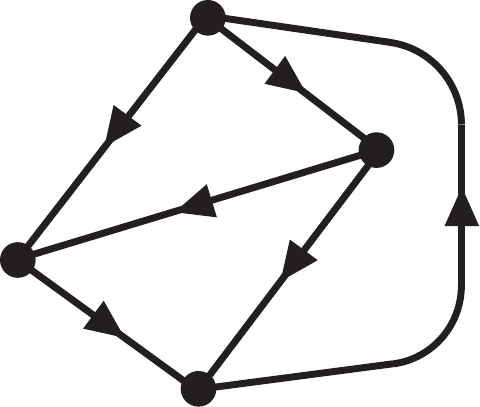}
\put(7, 35){$a$}
\put(22, 14){$b$}
\put(6, 3){$c$}
\put(37, 11){$d$}
\put(48, 21){$e$}
\put(25, 32){$f$}
\end{overpic}} \right>_{\mbox{\rm CM}}
=\left[ \begin{array}{c} 2f+n\\ 2f+1 \end{array} \right]\left\{ \begin{array}{ccc}  a & b & c \\ d & e & f \\ \end{array} \right\}.
\]
From isotopies and rotations of the tetrahedron graph, we have symmetry relations of $\{\,\cdot\,\}_{tet}$. We show one of them for later use. Other relations are in Appendix \ref{app01}.
\begin{equation} \label{sym}
\left\{ \begin{array}{ccc}  a & b & c \\ d & e & f \\ \end{array} \right\}_{tet}
=\left\{ \begin{array}{ccc}  \overline{d} & \overline{b} & \overline{f} \\ \overline{a} & \overline{e} & \overline{c} \\ \end{array} \right\}_{tet}.
\end{equation}
CM invariants can be calculated axiomatically using following formulas for changes of diagrams:
\begin{equation} \label{eq1}
\left< \hspace{0.2cm} \raisebox{-17 pt}{\begin{overpic}[bb=0 0 48 114, height= 40 pt]{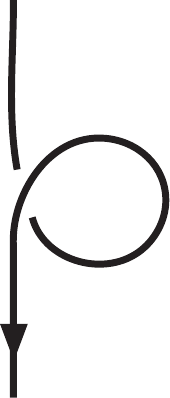}\put(5, 4){$a$}\end{overpic}}\right>_{\hspace{-0.1cm}\mbox{\rm CM}}
=\xi_n^{-2a\overline{a}}\left< \hspace{0.3cm} \raisebox{-17 pt}{\begin{overpic}[bb=0 0 8 114, height= 40 pt]{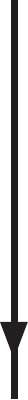}\put(5, 4){$a$}\end{overpic}}\hspace{0.3cm}\right>_{\hspace{-0.1cm}\mbox{\rm CM}}, 
\hspace{0.4cm}
\left< \hspace{0.2cm} \raisebox{-17 pt}{\begin{overpic}[bb=0 0 49 114, height= 40 pt]{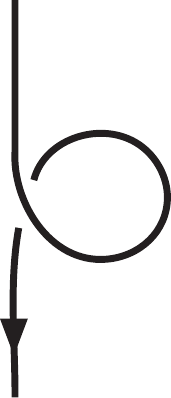}\put(5, 4){$a$}\end{overpic}}\right>_{\hspace{-0.1cm}\mbox{\rm CM}}
=\xi_n^{2a\overline{a}}\left< \hspace{0.3cm} \raisebox{-17 pt}{\begin{overpic}[bb=0 0 8 114, height= 40 pt]{Reidemaister01-3.pdf}\put(5, 4){$a$}\end{overpic}}\hspace{0.3cm}\right>_{\hspace{-0.1cm}\mbox{\rm CM}},
\end{equation}
\vspace{0.1cm}
\begin{equation} \label{eq2}
\left< \hspace{0.1cm} \raisebox{-17 pt}{\begin{overpic}[bb=0 0 72 84, height= 40 pt]{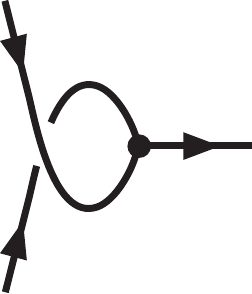}\put(5, 2){$a$}\put(5, 31){$b$}\put(25, 24){$c$}\end{overpic}}\right>_{\hspace{-0.1cm}\mbox{\rm CM}}
\hspace{-0.5cm}=\xi_n^{a\overline{a}+b\overline{b}-c\overline{c}}\left< \hspace{0.1cm} \raisebox{-17 pt}{\begin{overpic}[bb=0 0 72 85, height= 40 pt]{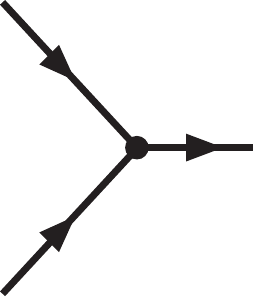}\put(9, 2){$a$}\put(10, 31){$b$}\put(25, 24){$c$}\end{overpic}}\right>_{\hspace{-0.1cm}\mbox{\rm CM}}, 
\hspace{0.3cm}
\left< \hspace{0.1cm} \raisebox{-17 pt}{\begin{overpic}[bb=0 0 73 84, height= 40 pt]{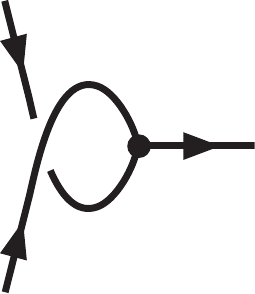}\put(5, 2){$a$}\put(5, 31){$b$}\put(25, 24){$c$}\end{overpic}}\right>_{\hspace{-0.1cm}\mbox{\rm CM}}
\hspace{-0.5cm}=\xi_n^{-a\overline{a}-b\overline{b}+c\overline{c}}\left< \hspace{0.1cm} \raisebox{-17 pt}{\begin{overpic}[bb=0 0 72 85, height= 40 pt]{Reidemaister04-4.pdf}\put(9, 2){$a$}\put(10, 31){$b$}\put(25, 24){$c$}\end{overpic}}\right>_{\hspace{-0.1cm}\mbox{\rm CM}},
\end{equation}
\vspace{0.1cm}
\begin{equation}
\left< \hspace{0.0cm} \raisebox{-18 pt}{\begin{overpic}[bb=0 0 176 184, height= 39 pt]{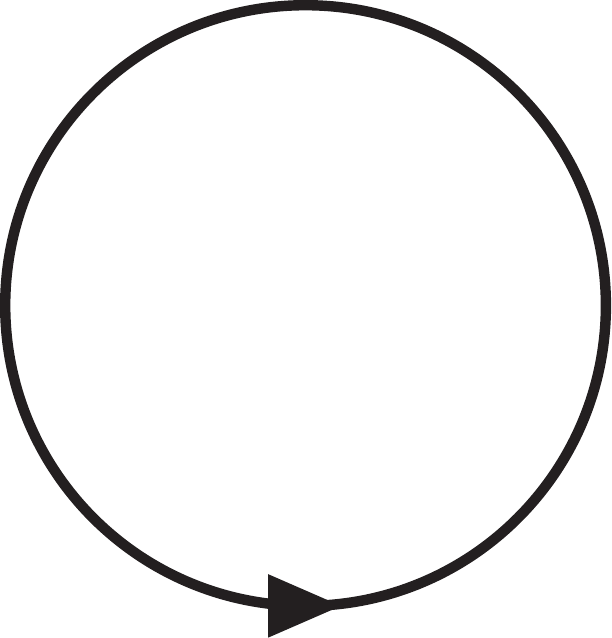}\put(15, 6){$a$}\end{overpic}}\right>_{\hspace{0.0cm}\mbox{\rm CM}}
=\left[ \begin{array}{c} 2a+n\\ 2a+1 \end{array} \right]^{-1},
\hspace{0.5cm}
\left<  \raisebox{-19 pt}{\begin{overpic}[bb=0 0 184 129, width=53 pt]{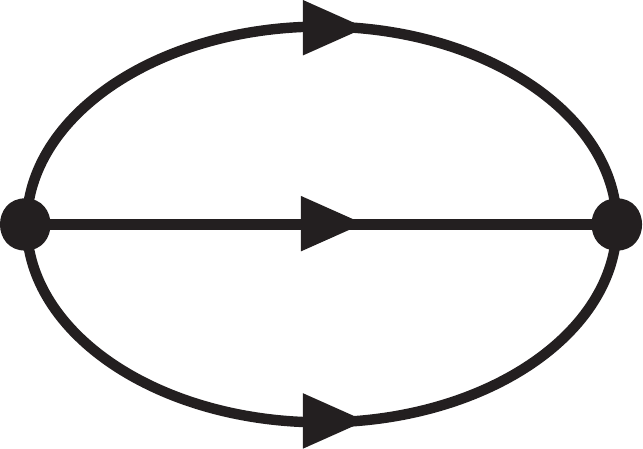}\put(24, 39){$a$}\put(24,22){$b$}\put(24,6){$c$}\end{overpic}}\right>_{\mbox{\rm CM}}
=1, 
\end{equation}
\vspace{0.1cm}
\begin{equation} \label{eq3}
\left<  \raisebox{-17 pt}{\begin{overpic}[bb=0 0 106 91, width=45 pt]{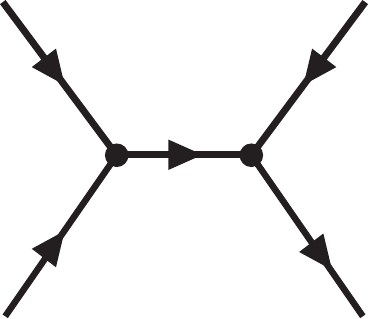}\put(8, 5){$a$}\put(8,30){$b$}\put(20,23){$e$}\put(31,30){$c$}\put(29,5){$d$}\end{overpic}} \right>_{\mbox{\rm CM}}
=
\sum_f \left\{ \begin{array}{ccc}  a & b & e \\ c & d & f \\ \end{array} \right\} \left<\raisebox{-17 pt}{\begin{overpic}[bb=0 0 91 88, width=42 pt]{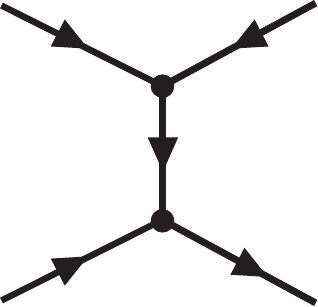}\put(5, 8){$a$}\put(5,26){$b$}\put(32,28){$c$}\put(32, 8){$d$}\put(24, 18){$f$}\end{overpic}} \right>_{\mbox{\rm CM}},
\end{equation}
\vspace{0.1cm}
\begin{equation} \label{eq5}
\left<  \raisebox{-19 pt}{\begin{overpic}[bb=0 0 247 213, width=45 pt]{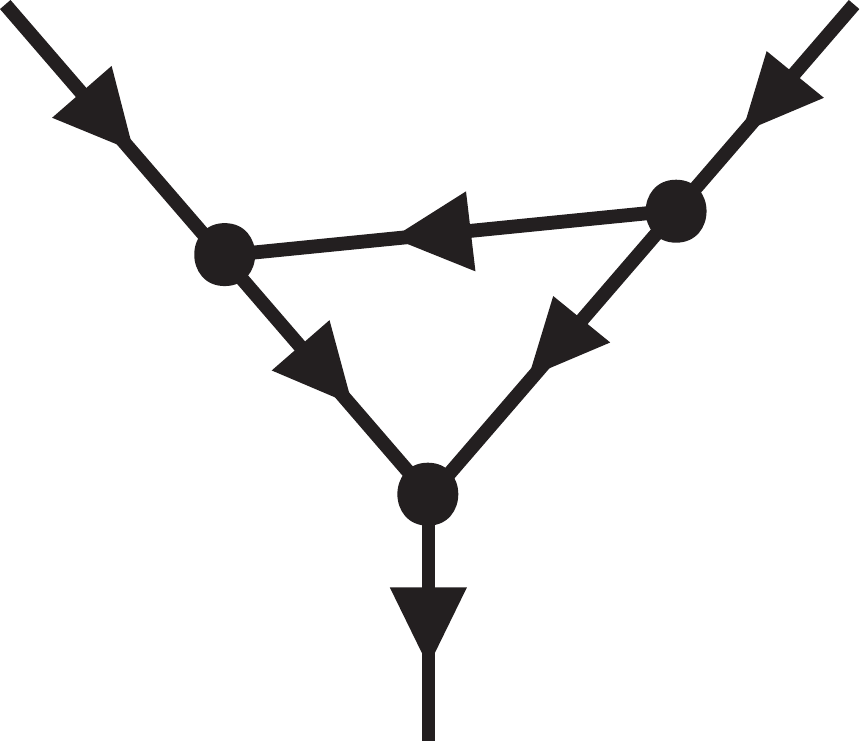}\put(0, 25){$a$}\put(20,30){$b$}\put(8,15){$c$}\put(31, 15){$d$}\put(25, 5){$e$}\put(40,25){$f$}\end{overpic}} \right>_{\mbox{\rm CM}}
=
\left\{ \begin{array}{ccc}  a & b & c \\ d & e & f \\ \end{array} \right\}_{tet}\left<\raisebox{-18 pt}{\begin{overpic}[bb=0 0 246 214, width=45 pt]{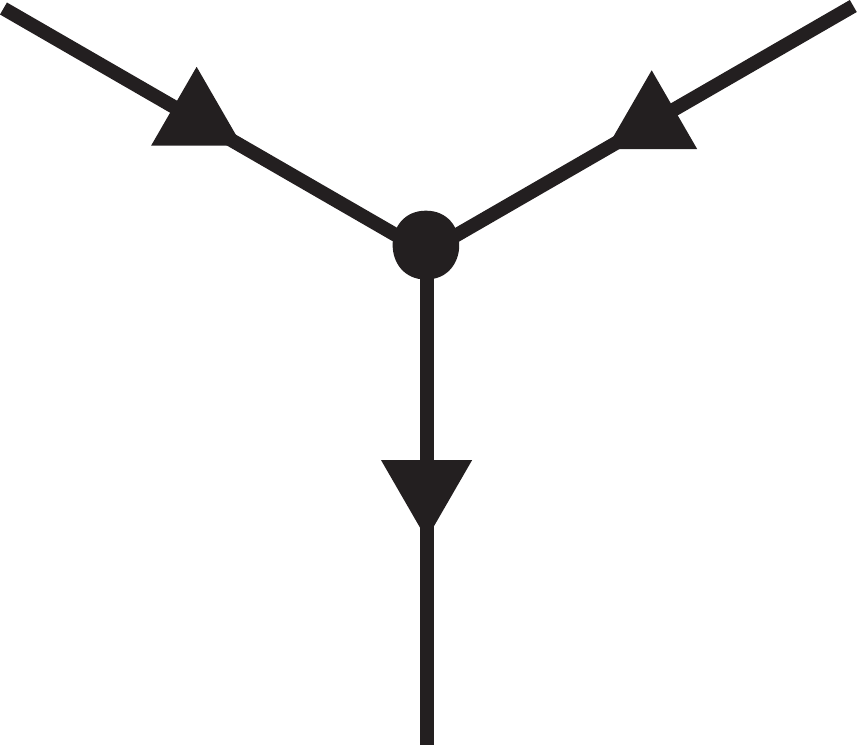}\put(5, 25){$a$}\put(26,12){$e$}\put(35,24){$f$}\end{overpic}} \right>_{\mbox{\rm CM}},
\end{equation}
\vspace{0.1cm}
\begin{equation}
\left< \hspace{0.3cm} \raisebox{-16 pt}{\begin{overpic}[bb=0 0 59 100, height= 38 pt]{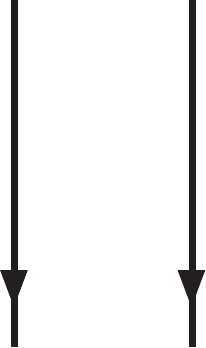}\put(-6, 5){$a$}\put(24, 5){$b$}\end{overpic}} \hspace{0.3cm} \right>_{\mbox{\rm CM}}
=\sum_{c}\left[ \begin{array}{c} 2c+n\\ 2c+1 \end{array} \right]^{-1}\left< \hspace{0.3cm}  \raisebox{-16 pt}{\begin{overpic}[bb=0 0 53 100, height= 38 pt]{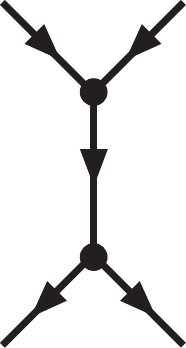}\put(-5, 3){$a$}\put(-5, 30){$a$}\put(20, 3){$b$}\put(20, 30){$b$}\put(12, 18){$c$}\end{overpic}}\hspace{0.3cm} \right>_{\mbox{\rm CM}}, 
\end{equation}
\vspace{0.1cm}
\begin{equation}
\left< \hspace{0.2cm} \raisebox{-17 pt}{\begin{overpic}[bb=0 0 42 101, height= 40 pt]{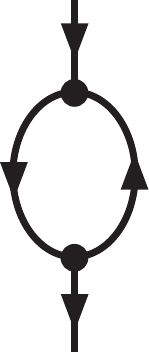}\put(11,1){$d$}\put(11, 34){$a$}\put(-6, 17){$b$}\put(18, 17){$c$}\end{overpic}} \hspace{0.2cm} \right>_{\mbox{\rm CM}}
=\delta_{ad}
\left[ \begin{array}{c} 2a+n\\ 2a+1 \end{array} \right] \left< \hspace{0.3cm} \raisebox{-17 pt}{\begin{overpic}[bb=0 0 8 114, height= 40 pt]{Reidemaister01-3.pdf}\put(5, 4){$a$}\end{overpic}}\hspace{0.3cm}\right>_{\mbox{\rm CM}}.
\end{equation}

\begin{rem}
Let $\Gamma$ be a framed oriented colored trivalent graph. The CM invariant of the mirror image of $\Gamma$ is the complex conjugate of the CM invariant of $\Gamma$ due to (\ref{eq1}) and (\ref{eq2}). Since a plane graph is isotopic to its mirror image, the CM invariant of a plane graph is a real number.
\end{rem}

\begin{rem}
For a half-integer $a \in \frac{1}{2}\mathbb{Z}$, $\left[ \begin{array}{c} 2a+n\\ 2a+1 \end{array} \right] = 0.$ The CM invariants have a factor $\left[ \begin{array}{c} 2a+n\\ 2a+1 \end{array} \right]^{-1}$ for every color $a$ of edges. Thus, if we attach a half-integer color to an edge of a graph formally, CM invariants may not be determined. It was proved in \cite{CM}, however, that $\{\,\cdot\,\}_{tet}$ is well-defined even for half-integer colors.
\end{rem}

\subsection{Volume conjecture of CM invariants}
\par
It was proved that CM invariants of tetrahedron graphs are related to the hyperbolic volumes of \textit{ideal} and \textit{truncated} tetrahedra (Figure \ref{fig5} left; for details see \cite{U} for example). The shape of a tetrahedron in the hyperbolic space is determined by its six dihedral angles. An ideal tetrahedron is a hyperbolic tetrahedron whose vertices are all at infinity points of the hyperbolic space. A pair of opposite edges of ideal tetrahedra has the same dihedral angle. Therefore an shape of an ideal tetrahedron is determined by the three dihedral angles $\alpha, \beta, \gamma$. It is known that they satisfy $\alpha+\beta+\gamma=\pi$. In the Klein model of hyperbolic space, we can consider tetrahedron whose vertices are at the ``outside" of the hyperbolic space (Figure \ref{fig5} right). For every vertex of this tetrahedron, there is just one geodesic surface which intersects perpendicularly to each of the three faces around the vertex. Cutting the tetrahedron by the surfaces at all vertices, we have a finite polyhedron in the hyperbolic space. This polyhedron is called a truncated tetrahedron. The three dihedral angles $\alpha, \beta, \gamma$ at the edges adjacent to an ``outside" vertex satisfy $\alpha+\beta+\gamma<\pi$.
\begin{figure}[ht]
$$\raisebox{-30 pt}{\includegraphics[bb=0 0 133 133, width=65 pt]{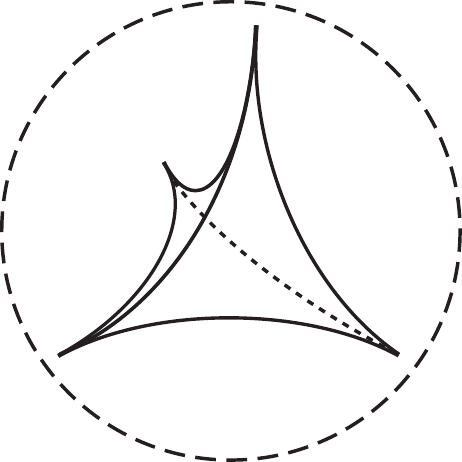}}\,\,\, \hspace{0.5cm}\raisebox{-30 pt}{\includegraphics[bb=0 0 88 88, height=65 pt]{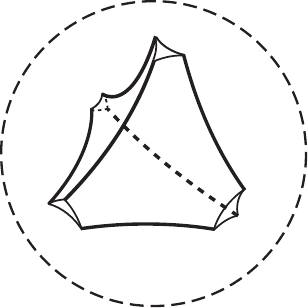}}\,\,\,,\hspace{0.5cm} \raisebox{-20 pt}{\begin{overpic}[bb=0 0 172 134, width=80 pt]{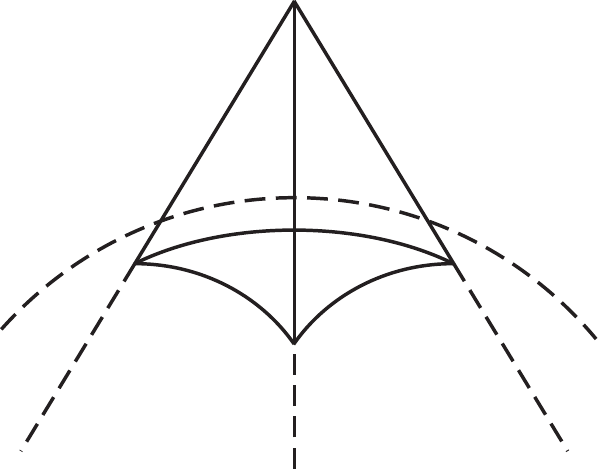}\put(10,-14){Klein Model}\end{overpic}} \,\,\, \mbox{\LARGE$\leftrightarrow$} \,\,\, \raisebox{-20 pt}{\begin{overpic}[bb=0 0 172 90, width=90 pt]{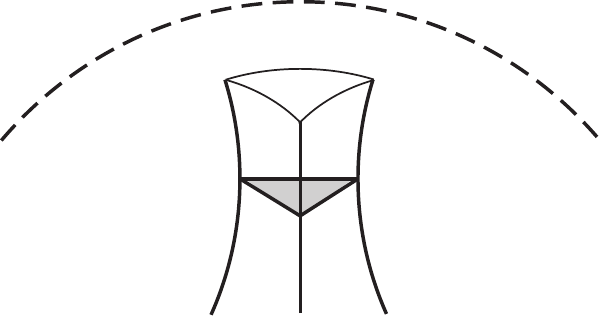}\put(12, -14){Poincar\'{e} Model}\end{overpic}}$$
\caption{Rough images of ideal and truncated tetrahedra (left), a vertex outside the hyperbolic space and a cutting surface (gray) (right).} \label{fig5}
\end{figure}
\begin{thm}[\cite{CM}] \label{thm3-1}
Let $S$ be a hyperbolic tetrahedron, let $\theta_a, \cdots , \theta_f$ be dihedral angles of $S$ and let $a_n, \cdots , f_n$ be sequences of integer colors such that $\lim_{n\rightarrow \infty} \frac{2\pi a_n}{n} = \pi - \theta_a, \cdots,$ $ \lim_{n\rightarrow \infty} \frac{2\pi f_n}{n} = \pi - \theta_f$. The CM invariants of tetrahedron graphs $\{\,\cdot\,\}_{tet}$ with the corresponding colors to the dihedral angles of $S$ satisfy the following formulas: if $S$ is ideal (i.e. the dihedral angles of the opposite edges are equal), 
\allowdisplaybreaks\begin{eqnarray*}
 {\rm Vol}(S)&=&\lim_{n\rightarrow \infty}\frac{\pi}{n}\log\left((-1)^{n-1} \left\{ \begin{array}{ccc}  a_n & b_n & c_n \\ a_n & b_n & c_n \\ \end{array} \right\}_{tet}  \right)\\
 &=& \lim_{n\rightarrow \infty}\frac{\pi}{n}\log\left((-1)^{n-1} \left\{ \begin{array}{ccc}  \overline{a_n} & \overline{b_n} & \overline{c_n} \\ \overline{a_n} & \overline{b_n} & \overline{c_n} \\ \end{array} \right\}_{tet}  \right).
\end{eqnarray*}
If $S$ is a truncated tetrahedron, 
\begin{equation} \label{eq04} 
{\rm Vol}(S)=\lim_{n\rightarrow \infty}\frac{\pi}{2n}\log\left( \left\{ \begin{array}{ccc}  a_n & b_n & c_n \\ d_n & e_n & f_n \\ \end{array} \right\}_{tet}  \left\{ \begin{array}{ccc}  \overline{a_n} & \overline{b_n} & \overline{c_n} \\ \overline{d_n} & \overline{e_n} & \overline{f_n} \\ \end{array} \right\}_{tet} \right).
\end{equation}
\end{thm}

\begin{rem} \label{rem01}
\par 
The value inside the $\log (\cdot)$ of (\ref{eq04}) does not depend on the orientations of the tetrahedral graph by the following lemma and the symmetry relations of $\{\,\cdot\,\}_{tet}$ (see Appendix \ref{app01}).
\begin{lem} \label{prop01} For integer colors $a, b, c, d, e, f$, the following holds:
$$
\left\{ \begin{array}{ccc}  a & b & c \\ d & e & f \\ \end{array} \right\}_{tet} \left\{ \begin{array}{ccc} \overline{a} & \overline{b} & \overline{c} \\ \overline{d} & \overline{e} & \overline{f} \\ \end{array} \right\}_{tet} 
=
\left\{ \begin{array}{ccc}  a & \overline{b} & c \\ d & e & f \\ \end{array} \right\}_{tet} \left\{ \begin{array}{ccc} \overline{a} & b & \overline{c} \\ \overline{d} & \overline{e} & \overline{f} \\ \end{array} \right\}_{tet}.
$$
\end{lem}
\begin{proof} See Appendix \ref{app02}.
\end{proof}
\end{rem}

\begin{rem} 
\par 
The relation (\ref{eq5}) corresponds to cutting a tetrahedron from a polyhedron with a truncation surface (Figure \ref{fig04}). Thus for convex polyhedra which are made of tetrahedra by gluing their truncated surfaces, the similar formula to (11) hold, which relates the CM invariants of their one-skeleton graphs to their hyperbolic volumes.
\end{rem}

\begin{figure}[ht]
$$
\left<\hspace{0.0cm} \raisebox{-23 pt}{\begin{overpic}[bb=0 0 188 156, width=60 pt]{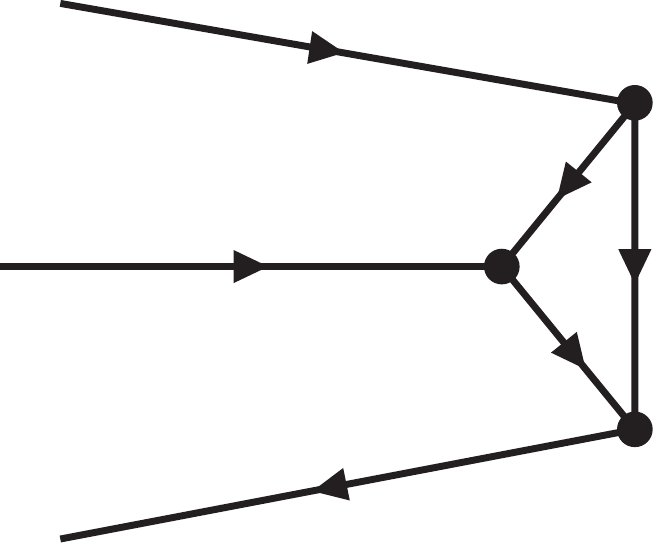}
\end{overpic}}  \hspace{0.0cm} \, \right>_{\!\!\raisebox{4 pt}{\mbox{\rm CM}}}
\!\!=
\left<\hspace{0.0cm} \raisebox{-23 pt}{\begin{overpic}[bb=0 0 85 156, width=27 pt]{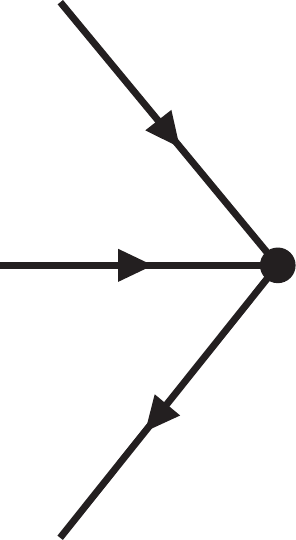}
\end{overpic}}  \hspace{0.0cm} \right>_{\!\!\raisebox{4 pt}{\mbox{\rm CM}}} 
\left<\hspace{0.0cm} \raisebox{-14 pt}{\begin{overpic}[bb=0 0 93 104, width=29 pt]{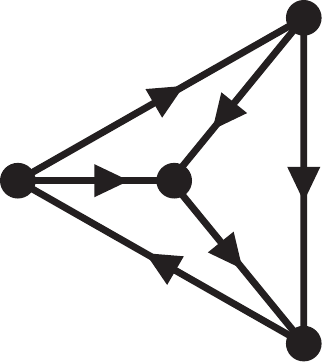}
\end{overpic}}  \hspace{0.0cm} \, \right>_{\!\!\raisebox{-5 pt}{\mbox{\rm CM}}}\!,
\hspace{0.3cm}
\raisebox{-15 pt}{\begin{overpic}[bb=0 0 177 85, height=37 pt]{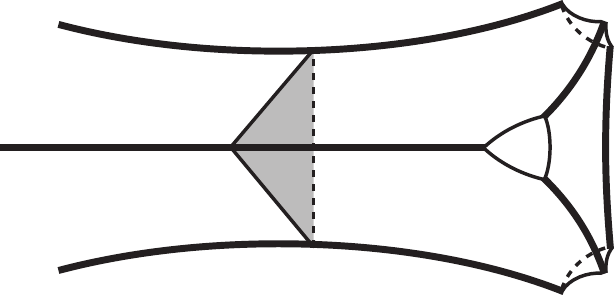}
\end{overpic}}
\,\,\,\mbox{\Large$\rightarrow$}\,\,\,
\raisebox{-18 pt}{\begin{overpic}[bb=0 0 220 85, height=45 pt]{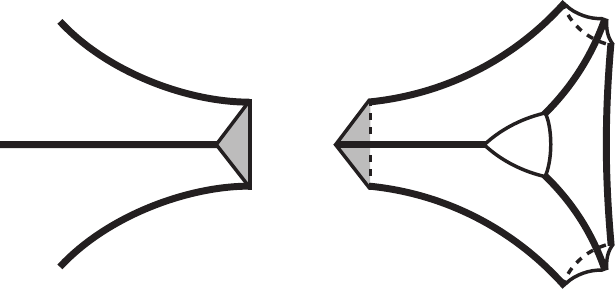}
\end{overpic}}
$$
\caption{A correspondence between the relation (\ref{eq5}) and cutting off a tetrahedron.} \label{fig04}
\end{figure}

\section{Yokota type invariants} \label{sec3}
\par
In this section, we introduce Yokota type invariants which are invariants of oriented colored spatial graphs whose valencies are more than or equal to three. These invariants are constructed from the CM invariants. We also propose a volume conjecture for the Yokota type invariants of plane graphs which relates to the volumes of hyperbolic convex polyhedra.
\subsection{Definition of Yokota type invariants}
\par
Using the similar way to define Yokota's invariants \cite{Yo} (see also \cite{BGM}, \cite{Ye}), the CM invariants are generalized to invariants for non-framed oriented colored spatial graphs whose valencies are more than or equal to three. These invariants are first defined for trivalent graphs and then generalized for graphs with vertices whose valencies are more than three.
\begin{defn}[Yokota type invariants]
Let $\Gamma$ be an oriented colored spatial graph and let $D$ be its diagram. If every vertex of $\Gamma$ is trivalent, then we define 
$$\left<\!\left< \Gamma \right>\!\right>_{\rm CM} = \left< D \right>_{\rm CM}\left< \overline{D} \,^r \right>_{\rm CM}, $$
where $\overline{\,\,\cdot\,\,}$ means the mirror image and $\cdot\,^{r}$ means reversing the orientations of all edges. Using the following relation, we generalize the definition of $\left<\!\left< \,\cdot\, \right>\!\right>_{\rm CM}$ for graphs with vertices whose valencies are more than three:
\begin{equation} \label{eqdef00} 
  \left<\!\!\!\left< \, \raisebox{-19 pt}{\includegraphics[bb=0 0 500 200,width=70 pt]{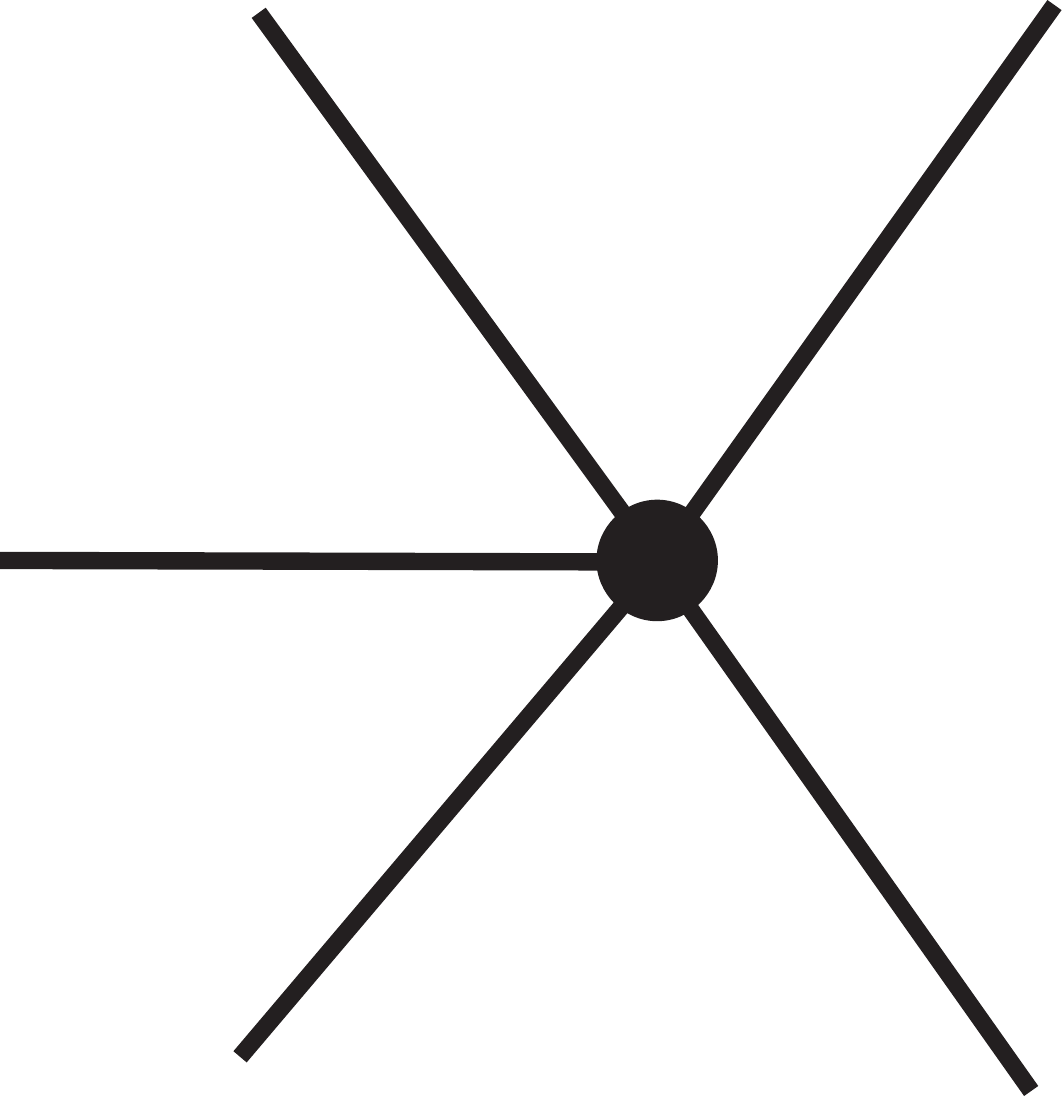}} \hspace{-0.8cm} \right>\!\!\!\right>_{\mbox{\rm CM}}
  =
  \sum_c \left[ \begin{array}{c} 2 c + n\\ 2 c + 1 \end{array} \right]^{-1}
  \left<\!\!\!\left<\raisebox{-19 pt}{\begin{overpic}[bb=0 0 133 83, width=70 pt]{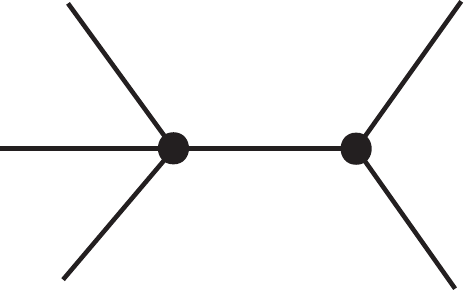}\put(37,24){$c$}\end{overpic}} \,\,\, \right>\!\!\!\right>_{\mbox{\rm CM}},
\end{equation}
where the surrounding edges have the same colors and orientations on the both sides and the orientation of the $c$ colored edge is arbitrary. We call the value $\left<\!\left< \,\cdot \,\right>\!\right>_{\rm CM}$ \textit{a Yokota type invariant}. 
\end{defn}
\par
Due to the relation (\ref{lem01}) in the following lemma, the Yokota type invariants are well-defined for non trivalent cases. For a vertex whose valency is more than three, expanding the vertex by using (\ref{eqdef00}) recursively, it changes to a trivalent tree. The shape of the tree depends on the way to extend the vertex. The values of the result graphs are, however, the same because the trees are transformed to each other by a sequence of the move in the relation (\ref{lem01}). The similar relation for Yokota's invariants is shown in \cite{Ye}.
\begin{lem} \label{lemB-01}
The following relation holds for the Yokota type invariants.
\begin{equation}  \label{lem01}
\sum_e \left[ \begin{array}{c} 2 e + n\\ 2 e + 1 \end{array} \right]^{-1} \left<\!\!\!\left<  \raisebox{-17 pt}{\begin{overpic}[bb=0 0 106 91, width=45 pt]{IHmove06-3.pdf}\put(8, 5){$a$}\put(8,30){$b$}\put(20,23){$e$}\put(31,30){$c$}\put(29,5){$d$}\end{overpic}} \right>\!\!\!\right>_{\mbox{\rm CM}}
=\sum_f \left[ \begin{array}{c} 2 f + n\\ 2 f + 1 \end{array} \right]^{-1} \left<\!\!\!\left<\raisebox{-17 pt}{\begin{overpic}[bb=0 0 91 88, width=42 pt]{IHmove01-3.pdf}\put(5, 8){$a$}\put(5,26){$b$}\put(32,28){$c$}\put(32, 8){$d$}\put(24, 18){$f$}\end{overpic}} \right>\!\!\!\right>_{\mbox{\rm CM}}.
\end{equation}
\end{lem}
\begin{proof}
We prepare the following two formulas by using (\ref{sym}) and (\ref{eq0}):
\begin{eqnarray*}
& &\hspace{-0.4cm}\bullet \hspace{0.4cm} \sum_g \left\{ \begin{array}{ccc}  \overline{a} & \overline{b} & \overline{e} \\ \overline{c} & \overline{d} & g \\ \end{array} \right\}  \left<\raisebox{-17 pt}{\begin{overpic}[bb=0 0 91 88, width=42 pt]{IHmove01-3.pdf}\put(4, 8){$\overline{a}$}\put(4,23){$\overline{b}$}\put(32,26){$\overline{c}$}\put(32, 8){$\overline{d}$}\put(24, 19){$g$}\end{overpic}} \right>_{\mbox{\rm CM}}\\
\\
& &\hspace{0.0cm}=\sum_g \left[ \begin{array}{c} 2 g + n\\ 2 g + 1 \end{array} \right]^{-1} \left\{ \begin{array}{ccc}  \overline{a} & \overline{b} & \overline{e} \\ \overline{c} & \overline{d} & g \\ \end{array} \right\}_{tet} \left<\raisebox{-17 pt}{\begin{overpic}[bb=0 0 91 88, width=42 pt]{IHmove01-3.pdf}\put(4, 8){$\overline{a}$}\put(4,23){$\overline{b}$}\put(32,26){$\overline{c}$}\put(32, 8){$\overline{d}$}\put(24, 19){$g$}\end{overpic}} \right>_{\mbox{\rm CM}}\\
\\
& &\hspace{0.0cm}=\sum_g \left[ \begin{array}{c} 2 g + n\\ 2 g + 1 \end{array} \right]^{-1} \left\{ \begin{array}{ccc}  c & b & \overline{g} \\ a & d & e \\ \end{array} \right\}_{tet} \left<\raisebox{-17 pt}{\begin{overpic}[bb=0 0 91 88, width=42 pt]{IHmove01-3.pdf}\put(4, 8){$\overline{a}$}\put(4,23){$\overline{b}$}\put(32,26){$\overline{c}$}\put(32, 8){$\overline{d}$}\put(24, 19){$g$}\end{overpic}} \right>_{\mbox{\rm CM}}\\
\\
& &\hspace{0.0cm}=\sum_g \left[ \begin{array}{c} 2 g + n\\ 2 g + 1 \end{array} \right]^{-1} \left[ \begin{array}{c} 2 e + n\\ 2 e + 1 \end{array} \right] \left\{ \begin{array}{ccc}  c & b & \overline{g} \\ a & d & e \\ \end{array} \right\} \left<\raisebox{-17 pt}{\begin{overpic}[bb=0 0 91 88, width=42 pt]{IHmove01-3.pdf}\put(4, 8){$\overline{a}$}\put(4,23){$\overline{b}$}\put(32,26){$\overline{c}$}\put(32, 8){$\overline{d}$}\put(24, 19){$g$}\end{overpic}} \right>_{\mbox{\rm CM}}.\\
\end{eqnarray*}
\begin{eqnarray*}
& &\hspace{0.0cm} \bullet \hspace{0.4cm} \left[ \begin{array}{c} 2 \overline{f} + n\\ 2 \overline{f}  + 1 \end{array} \right] = \left[ \begin{array}{c} 2(n-1-f) + n\\ 2(n-1-f) + 1 \end{array} \right] = \left[ \begin{array}{c} n-1-2f-1+2n\\ n-1-2f-n+2n \end{array} \right]\\
\\
& &\hspace{0.4cm}=\left((-1)^{n-1}\right)^{2} \left[ \begin{array}{c} n-1-2f-1\\ n-1-2f-n \end{array} \right]= \left[ \begin{array}{c} 2 f + n\\ 2 f + 1 \end{array} \right].
\end{eqnarray*}
Then we have
\allowdisplaybreaks\begin{eqnarray*}
& &\hspace{0.0cm}\sum_e \left[ \begin{array}{c} 2 e + n\\ 2 e + 1 \end{array} \right]^{-1} \left<\!\!\!\left<  \raisebox{-17 pt}{\begin{overpic}[bb=0 0 106 91, width=45 pt]{IHmove06-3.pdf}\put(8, 5){$a$}\put(8,30){$b$}\put(20,23){$e$}\put(31,30){$c$}\put(29,5){$d$}\end{overpic}} \right>\!\!\!\right>_{\mbox{\rm CM}}\\
\\
& &\hspace{-0.4cm}=\sum_e \left[ \begin{array}{c} 2 e + n\\ 2 e + 1 \end{array} \right]^{-1} \left<  \raisebox{-17 pt}{\begin{overpic}[bb=0 0 106 91, width=45 pt]{IHmove06-3.pdf}\put(8, 5){$a$}\put(8,30){$b$}\put(20,23){$e$}\put(31,30){$c$}\put(29,5){$d$}\end{overpic}} \right>_{\mbox{\rm CM}}\left<  \raisebox{-17 pt}{\begin{overpic}[bb=0 0 106 91, width=45 pt]{IHmove06-3.pdf}\put(8, 4){$\overline{a}$}\put(8,30){$\overline{b}$}\put(20,23){$\overline{e}$}\put(31,30){$\overline{c}$}\put(29,4){$\overline{d}$}\end{overpic}} \right>_{\mbox{\rm CM}}\\
\\
& &\hspace{-0.4cm}=\sum_e \left[ \begin{array}{c} 2 e + n\\ 2 e + 1 \end{array} \right]^{-1}\sum_f \left\{ \begin{array}{ccc}  a & b & e \\ c & d & f \\ \end{array} \right\}  \left<\raisebox{-17 pt}{\begin{overpic}[bb=0 0 91 88, width=42 pt]{IHmove01-3.pdf}\put(5, 8){$a$}\put(5,26){$b$}\put(32,28){$c$}\put(32, 8){$d$}\put(24, 18){$f$}\end{overpic}} \right>_{\mbox{\rm CM}}\sum_g \left\{ \begin{array}{ccc}  \overline{a} & \overline{b} & \overline{e} \\ \overline{c} & \overline{d} & g \\ \end{array} \right\}  \left<\raisebox{-17 pt}{\begin{overpic}[bb=0 0 91 88, width=42 pt]{IHmove01-3.pdf}\put(4, 8){$\overline{a}$}\put(4,23){$\overline{b}$}\put(32,26){$\overline{c}$}\put(32, 8){$\overline{d}$}\put(24, 19){$g$}\end{overpic}} \right>_{\mbox{\rm CM}}\\
\\
& &\hspace{-0.4cm}=\sum_f\sum_g \left[ \begin{array}{c} 2 g + n\\ 2 g + 1 \end{array} \right]^{-1}\sum_e \left\{ \begin{array}{ccc}  a & b & e \\ c & d & f \\ \end{array} \right\}\left\{ \begin{array}{ccc}  c & b & \overline{g} \\ a & d & e \\ \end{array} \right\}\left<\raisebox{-17 pt}{\begin{overpic}[bb=0 0 91 88, width=42 pt]{IHmove01-3.pdf}\put(5, 8){$a$}\put(5,26){$b$}\put(32,28){$c$}\put(32, 8){$d$}\put(24, 18){$f$}\end{overpic}} \right>_{\mbox{\rm CM}} \left<\raisebox{-17 pt}{\begin{overpic}[bb=0 0 91 88, width=42 pt]{IHmove01-3.pdf}\put(4, 8){$\overline{a}$}\put(4,23){$\overline{b}$}\put(32,26){$\overline{c}$}\put(32, 8){$\overline{d}$}\put(24, 19){$g$}\end{overpic}} \right>_{\mbox{\rm CM}}\\
\\
& &\hspace{-0.4cm}=\sum_f\sum_g \left[ \begin{array}{c} 2 g + n\\ 2 g + 1 \end{array} \right]^{-1} \delta_{f\overline{g}} \left<\raisebox{-17 pt}{\begin{overpic}[bb=0 0 91 88, width=42 pt]{IHmove01-3.pdf}\put(5, 8){$a$}\put(5,26){$b$}\put(32,28){$c$}\put(32, 8){$d$}\put(24, 18){$f$}\end{overpic}} \right>_{\mbox{\rm CM}} \left<\raisebox{-17 pt}{\begin{overpic}[bb=0 0 91 88, width=42 pt]{IHmove01-3.pdf}\put(4, 8){$\overline{a}$}\put(4,23){$\overline{b}$}\put(32,26){$\overline{c}$}\put(32, 8){$\overline{d}$}\put(24, 19){$g$}\end{overpic}} \right>_{\mbox{\rm CM}}\\
\\
& &\hspace{-0.4cm}=\sum_f \left[ \begin{array}{c} 2 \overline{f} + n\\ 2 \overline{f} + 1 \end{array} \right]^{-1} \left<\raisebox{-17 pt}{\begin{overpic}[bb=0 0 91 88, width=42 pt]{IHmove01-3.pdf}\put(5, 8){$a$}\put(5,26){$b$}\put(32,28){$c$}\put(32, 8){$d$}\put(24, 18){$f$}\end{overpic}} \right>_{\mbox{\rm CM}} \left<\raisebox{-17 pt}{\begin{overpic}[bb=0 0 91 88, width=42 pt]{IHmove01-3.pdf}\put(4, 8){$\overline{a}$}\put(4,23){$\overline{b}$}\put(32,26){$\overline{c}$}\put(32, 8){$\overline{d}$}\put(24, 17){$\overline{f}$}\end{overpic}} \right>_{\mbox{\rm CM}}\\
\\
& &\hspace{-0.4cm}=\sum_f \left[ \begin{array}{c} 2 f + n\\ 2 f + 1 \end{array} \right]^{-1} \left<\!\!\!\left<\raisebox{-17 pt}{\begin{overpic}[bb=0 0 91 88, width=42 pt]{IHmove01-3.pdf}\put(5, 8){$a$}\put(5,26){$b$}\put(32,28){$c$}\put(32, 8){$d$}\put(24, 18){$f$}\end{overpic}} \right>\!\!\!\right>_{\mbox{\rm CM}},
\end{eqnarray*}
where we use (\ref{eq3}) for the second identity, (\ref{orth}) for the fourth identity and the two relations above for the third and sixth identities respectively.
\end{proof}
\par
We prove that $\left<\!\left<\,\cdot\,\right>\!\right>_{\rm CM}$ is actually an invariant for oriented colored spatial graphs. 
\begin{thm}
The values of the Yokota type invariants are independent of the choice of the diagrams to calculate. 
Thus, $\left<\!\left<\,\cdot\,\right>\!\right>_{\rm CM}$ is an invariant of oriented colored spatial graphs whose valencies are more than or equal to three.
\end{thm}
\begin{proof}
It is well-known that two diagrams of a spatial graph are transformed to each other by a sequence of five Reidemeister moves RI-RV which are in Figure \ref{fig03}, where we assume that the surrounding edges of the both sides of each move have the same orientations and colors. We check the invariance of the Yokota type invariants for the Reidemeister moves. For the trivalent case, the invariance for the RII, RIII and RV moves comes from that of the CM invariants. The invariance for RI and RIV moves comes from direct calculations using the relations (\ref{eq1}) and (\ref{eq2}) respectively. For a graph with vertices whose valencies are more than three, by expanding the vertices, the graph becomes trivalent. Thus the invariance for the Reidemeister moves for this case is reduced to that of the trivalent case.
\end{proof}

\begin{figure}[ht]
\begin{centering}
 \includegraphics[bb= 0 0 93 53, height=48 pt]{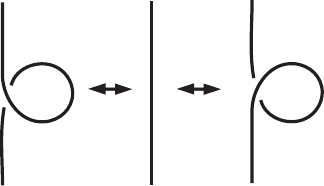}\hspace{1.3cm} \includegraphics[bb= 0 0 80 65, height=48 pt]{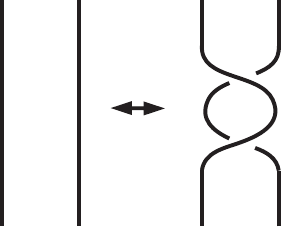} \hspace{1.3cm} \includegraphics[bb= 0 0 104 41, height=48 pt, width=90pt]{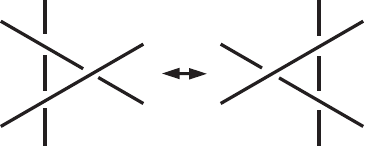}\\
 \vspace{0.2cm}
  \hspace{0.0cm} {RI} \hspace{3.3cm} {RII} \hspace{3.3cm} {RIII}\\
 \vspace{0.5cm}
   \hspace{0.0cm} \includegraphics[bb= 0 0 288 91, height=48 pt]{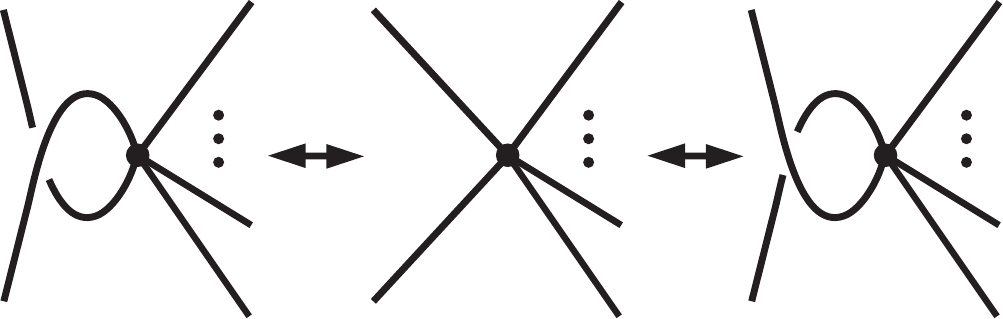} \hspace{1.1cm} \includegraphics[bb= 0 0 249 92, height=48 pt]{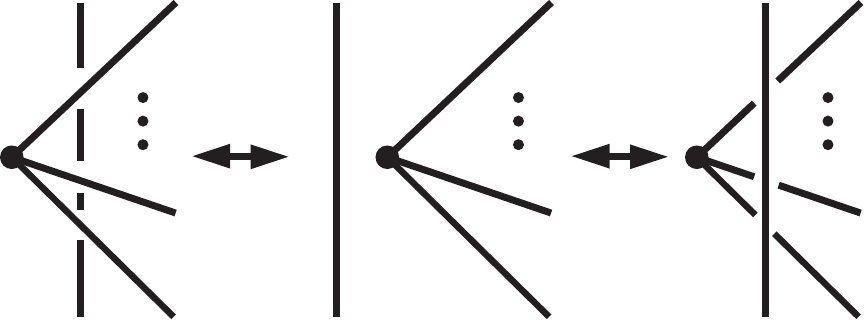}\\
 \vspace{0.2cm}
  \hspace{0.7cm} {RIV} \hspace{5.5cm} {RV} 
\end{centering}
\caption{The Reidemeister moves for spatial graphs.} \label{fig03}
\end{figure}

In Theorem \ref{thm3-1}, the value inside $\log(\,\cdot\,)$ of (\ref{eq04}) is the value of the Yokota type invariants for the tetrahedron graph. Using the Yokota type invariants, we conjecture an extension of Theorem \ref{thm3-1}.
\begin{con} \label{con2}
Let $\Gamma$ be an oriented plane graph and let $S_{\Gamma}$ be a hyperbolic convex polyhedron whose one-skeleton is $\Gamma$. If sequences of integer colors of $\Gamma$ are taken as in Theorem \ref{thm3-1} for the corresponding dihedral angles of $S_{\Gamma}$, then
$$ {\rm Vol}(S_{\Gamma})=\lim_{n\rightarrow \infty}\frac{\pi}{2n}\log\left(\left| \left<\!\left< \Gamma \right>\!\right>_{\rm CM}\right| \right).$$
Here, we take the absolute values of the Yokota type invariants to omit minus signs in case they appear. 
\end{con}
In general, the Yokota type invariants are not well-defined for integer colors. We, however, expect that the Yokota type invariants of plane graphs are well-defined for integer colors (see Section \ref{sec4}). 

\section{Numerical calculations} \label{sec4}
 By numerical calculations at near-integer colors, we observe that the Yokota type invariants seem to be well-defined at integer colors. We also observe the asymptotic behaviors of them. The calculations in this chapter was done by using the software Mthematica.
\subsection{Numerical calculations}
We show numerical calculations of the Yokota type invariants for some square and pentagonal pyramid graphs and observe the asymptotic behaviors as $n \rightarrow \infty$. The value of the square and pentagonal pyramid graphs are calculated as follows:

\vspace{-0.1cm}
\allowdisplaybreaks\begin{eqnarray*} 
& &\hspace{-0.8cm}\left<\!\!\!\left< \hspace{0.1cm}\, \raisebox{-22 pt}{\begin{overpic}[bb=0 0 101 101, width=50 pt]{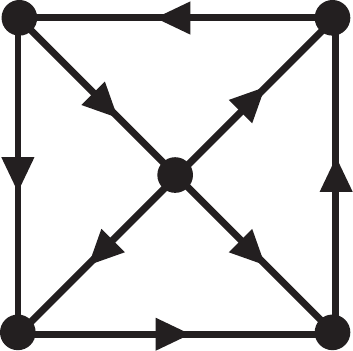}
\put(22, 52){${a}$}\put(51,21){${b}$}\put(22, -6){${c}$}\put(-6, 22){${d}$}
\put(37, 30){${e}$}\put(26, 9){${f}$}\put(9, 18){${g}$}\put(16, 37){${h}$}
\end{overpic}} \hspace{0.1cm}\,  \right>\!\!\!\right>_{\mbox{\rm CM}}
  =
  \sum_x \left[ \begin{array}{c} 2 x + n\\ 2 x + 1 \end{array} \right]^{-1}
  \left<\!\!\!\left< \,\hspace{0.1cm} \raisebox{-22 pt}{\begin{overpic}[bb=0 0 142 100, width=70 pt]{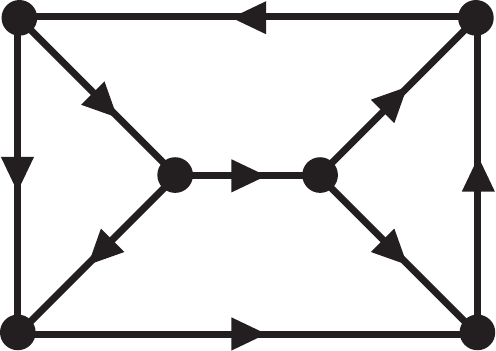}
\put(33, 51){${a}$}\put(71,21){${b}$}\put(32, -6){${c}$}\put(-6, 22){${d}$}
\put(57, 30){${e}$}\put(46, 9){${f}$}\put(9, 18){${g}$}\put(16, 37){${h}$}
\put(32, 29){${x}$}
\end{overpic}}  \hspace{0.1cm}\, \right>\!\!\!\right>_{\mbox{\rm CM}}\\
\\
\\
& &\hspace{0.0cm} =
  \sum_x \left[ \begin{array}{c} 2 x + n\\ 2 x + 1 \end{array} \right]^{-1}
  \left<\hspace{0.2cm} \raisebox{-22 pt}{\begin{overpic}[bb=0 0 142 100, width=70 pt]{graph03-4.pdf}
\put(33, 51){${a}$}\put(71,21){${b}$}\put(32, -6){${c}$}\put(-6, 22){${d}$}
\put(57, 30){${e}$}\put(46, 9){${f}$}\put(9, 18){${g}$}\put(16, 37){${h}$}
\put(32, 29){${x}$}
\end{overpic}}  \hspace{0.2cm} \right>_{\mbox{\rm CM}}\left<\hspace{0.2cm} \raisebox{-22 pt}{\begin{overpic}[bb=0 0 142 100, width=70 pt]{graph03-4.pdf}
\put(33,50){$\overline{a}$}\put(72, 21){$\overline{b}$}\put(32, -7){$\overline{c}$}\put(-6, 22){$\overline{d}$}
\put(57, 28){$\overline{e}$}\put(45, 8){$\overline{f}$}\put(9, 18){$\overline{g}$}\put(16, 36){$\overline{h}$}
\put(32, 29){$\overline{x}$}
\end{overpic}}  \hspace{0.2cm} \right>_{\mbox{\rm CM}}\\
\\
\\
& &\hspace{0.0cm} =
 \sum_x \left[ \begin{array}{c} 2 x + n\\ 2 x + 1 \end{array} \right]^{-1}
 \left\{ \begin{array}{ccc}  c & f & b \\ e & a & x \\ \end{array} \right\}_{tet} 
 \left\{ \begin{array}{ccc}  d & g & c \\ x & a & h \\ \end{array} \right\}_{tet}
 \left\{ \begin{array}{ccc}  \overline{c} & \overline{f} & \overline{b} \\ \overline{e} & \overline{a} & \overline{x} \\ \end{array} \right\}_{tet} 
 \left\{ \begin{array}{ccc}  \overline{d} & \overline{g} & \overline{c} \\ \overline{x} & \overline{a} & \overline{h} \\ \end{array} \right\}_{tet},
\end{eqnarray*}
where the third identity uses the relation (\ref{eq5}). Similarly, using the expansions at a vertex twice,
\allowdisplaybreaks\begin{eqnarray*} 
& &\hspace{-0.8cm}\left<\!\!\!\left< \hspace{0.1cm} \raisebox{-22 pt}{\begin{overpic}[bb=0 0 156 149, width=65 pt]{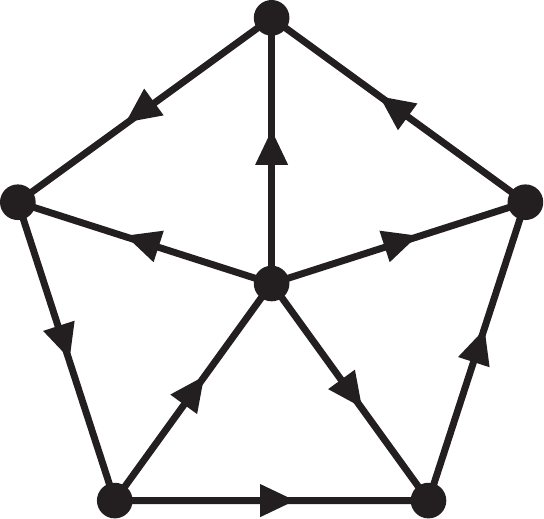}
\put(11,52){$a$}\put(49, 51){$b$}\put(60, 18){$c$}\put(29, -8){$d$}\put(-1, 18){$e$}
\put(24, 42){$p$}\put(41, 38){$q$}\put(43, 18){$r$}\put(25, 11){$s$}\put(14, 24){$t$}
\end{overpic}} \hspace{0.1cm}  \right>\!\!\!\right>_{\mbox{\rm CM}}
  =
  \sum_{x, y} \left[ \begin{array}{c} 2 x + n\\ 2 x + 1 \end{array} \right]^{-1} \left[ \begin{array}{c} 2 y + n\\ 2 y + 1 \end{array} \right]^{-1}
  \left<\!\!\!\left< \hspace{0.1cm} \raisebox{-22 pt}{\begin{overpic}[bb=0 0 156 149, width=65 pt]{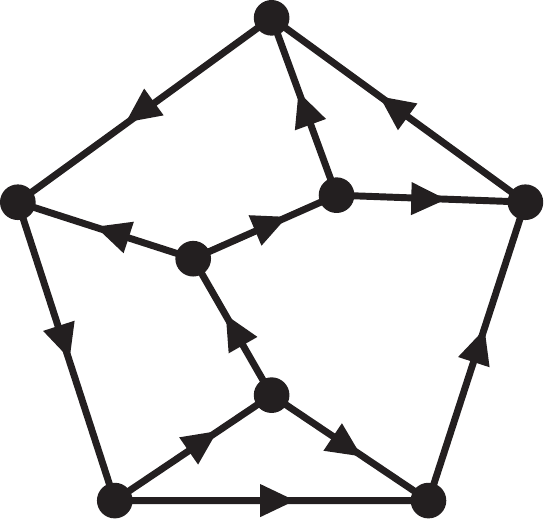}
\put(11,52){$a$}\put(49, 51){$b$}\put(60, 18){$c$}\put(29, -8){$d$}\put(-1, 18){$e$}
\put(29, 47){$p$}\put(47, 30){$q$}\put(42, 11){$r$}\put(19, 11){$s$}\put(12, 24){$t$}
\put(24.5, 36.5){${x}$}\put(31, 23){${y}$}
\end{overpic}}  \hspace{0.1cm} \right>\!\!\!\right>_{\mbox{\rm CM}}\\
\\
\\
& &\hspace{0.0cm} =
  \sum_{x, y} \left[ \begin{array}{c} 2 x + n\\ 2 x + 1 \end{array} \right]^{-1} \left[ \begin{array}{c} 2 y + n\\ 2 y + 1 \end{array} \right]^{-1}
  \left<\hspace{0.1cm} \raisebox{-22 pt}{\begin{overpic}[bb=0 0 156 149, width=65 pt]{pentaPyra-g02-4.pdf}
\put(11,52){$a$}\put(49, 51){$b$}\put(60, 18){$c$}\put(29, -8){$d$}\put(-1, 18){$e$}
\put(29, 47){$p$}\put(47, 30){$q$}\put(42, 11){$r$}\put(19, 11){$s$}\put(12, 24){$t$}
\put(24.5, 36.5){${x}$}\put(31, 23){${y}$}
\end{overpic}}  \hspace{0.1cm} \right>_{\mbox{\rm CM}}\left<\hspace{0.1cm} \raisebox{-22 pt}{\begin{overpic}[bb=0 0 156 149, width=65 pt]{pentaPyra-g02-4.pdf}
\put(11,52){$\overline{a}$}\put(49, 51){$\overline{b}$}\put(60, 17){$\overline{c}$}\put(30, -10){$\overline{d}$}\put(-1, 18){$\overline{e}$}
\put(29, 47){$\overline{p}$}\put(47, 29){$\overline{q}$}\put(42, 11){$\overline{r}$}\put(19, 11){$\overline{s}$}\put(12, 22){$\overline{t}$}
\put(24.5, 36.2){$\overline{x}$}\put(31, 23){$\overline{y}$}
\end{overpic}}  \hspace{0.1cm} \right>_{\mbox{\rm CM}}\\
\\
\\
& &\hspace{0.0cm} =
 \sum_{x, y} \left[ \begin{array}{c} 2 x + n\\ 2 x + 1 \end{array} \right]^{-1} \left[ \begin{array}{c} 2 y + n\\ 2 y + 1 \end{array} \right]^{-1}
 \left\{ \begin{array}{ccc}  c & q & b \\ p & a & x \\ \end{array} \right\}_{tet} 
 \left\{ \begin{array}{ccc}  c & x & a \\ t & e & y \\ \end{array} \right\}_{tet}
 \left\{ \begin{array}{ccc}  d & r & c \\ y & e & s \\ \end{array} \right\}_{tet}\\
& &\hspace{6.0cm} \times
 \left\{ \begin{array}{ccc}  \overline{c} & \overline{q} & \overline{b} \\ \overline{p} & \overline{a} & \overline{x} \\ \end{array} \right\}_{tet} 
 \left\{ \begin{array}{ccc}  \overline{c} & \overline{x} & \overline{a} \\ \overline{t} & \overline{e} & \overline{y} \\ \end{array} \right\}_{tet}
 \left\{ \begin{array}{ccc}  \overline{d} & \overline{r} & \overline{c} \\ \overline{y} & \overline{e} & \overline{s} \\ \end{array} \right\}_{tet}.
\end{eqnarray*}
The orientations of edges may not affect the values of the Yokota's invariants in the above formulas for the integer color case. This is because of Remark \ref{rem01} and the invariance of the ranges of the summations in the formulas for reversing the orientations of adjacent edges, see the first paragraph of Appendix \ref{app02}.
\par 
We consider oriented colored square pyramid graphs $\Gamma_{1,n}$, $\Gamma_{2,n}$ and oriented colored pentagonal pyramid graphs $\Gamma_{3, n}$, $\Gamma_{4, n}$ corresponding to the following hyperbolic polyhedra: 
\vspace{0.2cm}
$$ \hspace{0.4cm} \mbox{\Large$\Gamma_{1,n}: $} \hspace{0.1cm} \raisebox{-21 pt}{\begin{overpic}[bb=0 0 93 60, width=60 pt]{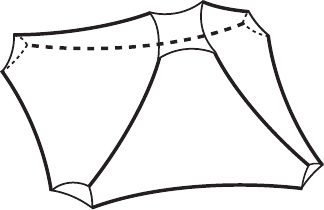}
\put(17, 23){$a$}\put(54,22){$b$}\put(34, -2){$c$}\put(-2, 10){$d$}
\put(40, 39){$e$}\put(38, 13){$f$}\put(26, 14){$g$}\put(14, 36){$h$}
\put(-1, -13){$a, b, c, d: \pi/4$}\put(-2, -26){$e, f, g, h: \pi/3$}
\end{overpic}}
\hspace{0.2cm}
\leftrightarrow 
\hspace{0.5cm}\raisebox{-22 pt}{\begin{overpic}[bb=0 0 101 101, width=50 pt]{graph03-3.pdf}
\put(20, 52){${a_n}$}\put(51,22){${b_n}$}\put(20, -6){${c_n}$}\put(-10, 23){${d_n}$}
\put(36, 29){${e_n}$}\put(25, 9){${f_n}$}\put(7, 20){${g_n}$}\put(16, 37){${h_n}$}
\end{overpic}} 
\hspace{0.5cm} \left\{ \begin{array}{ll}  a_n=3n/8 \, (+\,3\varepsilon) & b_n= 3n/8 \, (+\, 9 \varepsilon)\\  c_n=3n/8 \, (+\, 4 \varepsilon) & d_n=3n/8 \, (+\, \varepsilon) \\ e_n=n/3 \,\,\,\, (-\, 6 \varepsilon) & f_n=n/3 \,\,\,\, (+\, 5\varepsilon) \\ g_n=n/3 \,\,\,\, (+\,3 \varepsilon) & h_n=n/3 \,\,\,\, (+\, 2 \varepsilon) \end{array} \right. ,$$
\vspace{0.7cm}
$$ \mbox{\Large$\Gamma_{2,n}: $}\hspace{0.2cm}\raisebox{-21 pt}{\begin{overpic}[bb=0 0 103 63, width=63 pt]{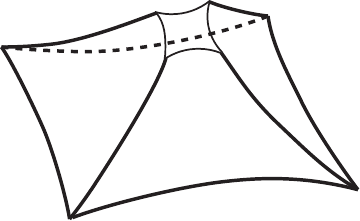}
\put(17, 23){$a$}\put(54,22){$b$}\put(35, 0){$c$}\put(0, 10){$d$}
\put(40, 39){$e$}\put(37, 14){$f$}\put(25, 15){$g$}\put(14, 35){$h$}
\put(-1, -11){$a, b, c, d$}\put(-2, -23){$e, f, g, h$}
\put(41,-17){$:\pi/3$}
\end{overpic}}
\leftrightarrow 
\hspace{0.5cm}\raisebox{-22 pt}{\begin{overpic}[bb=0 0 101 101, width=50 pt]{graph03-3.pdf}
\put(20, 52){${a_n}$}\put(51,22){${b_n}$}\put(20, -6){${c_n}$}\put(-10, 23){${d_n}$}
\put(36, 29){${e_n}$}\put(25, 9){${f_n}$}\put(7, 20){${g_n}$}\put(16, 37){${h_n}$}
\end{overpic}} 
\hspace{0.5cm} \left\{ \begin{array}{ll}  a_n=n/3 \,(+\,3\varepsilon) & b_n= n/3 \,(+ \, 9\varepsilon)\\  c_n=n/3\, (+ \, 4 \varepsilon) & d_n=n/3\, (+ \, \varepsilon) \\ e_n=n/3\, (- \, 6 \varepsilon) & f_n=n/3 \,(+\, 5\varepsilon) \\ g_n=n/3\, (+\,3 \varepsilon) & h_n=n/3\, (+\, 2 \varepsilon) \end{array} \right. ,$$
\vspace{0.8cm}
$$  \hspace{0.1cm}\mbox{\Large$\Gamma_{3,n}: $}\hspace{0.1cm}\raisebox{-21 pt}{\begin{overpic}[bb=0 0 116 74, width=68 pt]{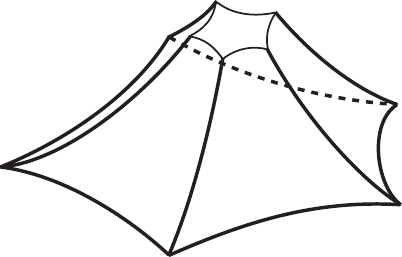}
\put(15, 29){$a$}\put(41, 22){$b$}\put(65, 17){$c$}\put(45, -1){$d$}\put(12, 3){$e$}
\put(30, 43){$p$}\put(55, 35){$q$}\put(49, 15){$r$}\put(28, 17){$s$}\put(21, 19){$t$}
\put(-1, -10){$a, b, c, d, e$}\put(-1, -22){$p, q, r, s, t$}
\put(48,-16){$:\pi/3$}
\end{overpic}}
\hspace{0.1cm}\leftrightarrow\hspace{0.2cm} 
\raisebox{-26 pt}{\begin{overpic}[bb=0 0 156 149, width=65 pt]{pentaPyra-g01-3.pdf}
\put(8,54){$a_n$}\put(49, 52){$b_n$}\put(60, 18){$c_n$}\put(28, -8){$d_n$}\put(-3, 18){$e_n$}
\put(21, 41){$p_n$}\put(41, 38){$q_n$}\put(43, 18){$r_n$}\put(25, 11){$s_n$}\put(14, 24){$t_n$}
\end{overpic}} 
\hspace{0.3cm} \left\{ \begin{array}{ll}  a_n=n/3 \,(+\,\varepsilon) & b_n= n/3 \,(+ \, 2 \varepsilon)\\  c_n=n/3\, (+ \, 3 \varepsilon) & d_n=n/3\, (+ \,  \varepsilon) \\ e_n=n/3\, (+ \, 2 \varepsilon) & p_n=n/3 \,(-\, \varepsilon) \\ q_n=n/3\, (-\, \varepsilon) & r_n=n/3\, (+\, 2 \varepsilon) \\ s_n=n/3\, (+\, \varepsilon) & t_n\,=n/3\, (+\,  \varepsilon) \end{array}\right. ,$$
\vspace{0.8cm}
$$ \hspace{0.4cm}\mbox{\Large$\Gamma_{4,n}: $}\hspace{0.1cm}\raisebox{-21 pt}{\begin{overpic}[bb=0 0 116 88, width=68 pt]{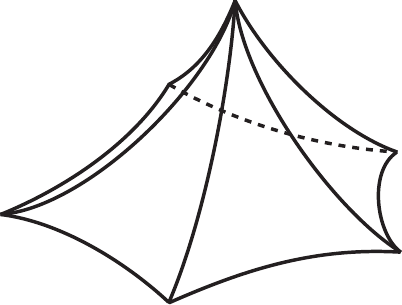}
\put(15, 29){$a$}\put(41, 22){$b$}\put(65, 17){$c$}\put(45, -1){$d$}\put(12, 3){$e$}
\put(30, 43){$p$}\put(55, 35){$q$}\put(49, 15){$r$}\put(28, 17){$s$}\put(21, 19){$t$}
\put(-1, -10){$a, b, c, d, e: \pi/5$}\put(-1, -22){$p, q, r, s, t: 3\pi/5$}
\end{overpic}}
\hspace{0.1cm}\leftrightarrow\hspace{0.2cm} 
\raisebox{-26 pt}{\begin{overpic}[bb=0 0 156 149, width=65 pt]{pentaPyra-g01-3.pdf}
\put(8,54){$a_n$}\put(49, 52){$b_n$}\put(60, 18){$c_n$}\put(28, -8){$d_n$}\put(-3, 18){$e_n$}
\put(21, 41){$p_n$}\put(41, 38){$q_n$}\put(43, 18){$r_n$}\put(25, 11){$s_n$}\put(14, 24){$t_n$}
\end{overpic}} 
\hspace{0.3cm} \left\{ \begin{array}{ll}  a_n=2n/5 \,(+\,\varepsilon) & b_n= 2n/5 \,(+ \, 2 \varepsilon)\\  c_n=2n/5\, (+ \, 3 \varepsilon) & d_n=2n/5\, (+ \,  \varepsilon) \\ e_n=2n/5\, (+ \, 2 \varepsilon) & p_n=n/5 \,\,\,\,(-\, \varepsilon) \\ q_n=n/5\,\,\,\, (-\, \varepsilon) & r_n=n/5\,\,\,\, (+\, 2 \varepsilon) \\ s_n=n/5\,\,\,\, (+\, \varepsilon) & t_n\,=n/5\,\,\,\, (+\,  \varepsilon) \end{array}\right. ,$$
\vspace{0.7cm}

\noindent where all vertices of the square pyramid of $\Gamma_{1,n}$ are truncated, the pyramids of $\Gamma_{2,n}$ and $\Gamma_{3,n}$ have truncated apexes and ideal bottom vertices, and the all vertices of the pentagonal pyramid of $\Gamma_{4,n}$ are ideal. The colors become integers when $n$ is an appropriate integer. For the integer colors,  the formulas of the square and pentagonal pyramid graphs are likely to diverge to infinity because of the coefficients $\left[ \begin{array}{c} 2 x + n\\ 2 x + 1 \end{array} \right]^{-1}$ and $\left[ \begin{array}{c} 2 y + n\\ 2 y + 1 \end{array} \right]^{-1}$. When we do numerical calculations, we perturb the integer colors by using a small real number $\varepsilon$ as above so that they satisfy admissible conditions at the trivalent vertices. 
\par
We did the numerical calculations at $\varepsilon=10^{-30}$ and observed the asymptotic behavior of the formulas for the pyramids. The results are in Table \ref{table1}, where we calculated until the order that our computer could and show the results to ninth decimal place. The hyperbolic volume of the square pyramid of $\Gamma_{1, n}$ is calculated by cutting it into the two same \textit{doubly truncated tetrahedra} and using the formula in \cite{KoMu}. The square pyramid of $\Gamma_{2, n}$ is the half of the hyperbolic regular cube and its volume is calculated by using the formula in \cite{Ma}. The hyperbolic pentagonal pyramids of $\Gamma_{3, n}$ and $\Gamma_{4, n}$ can be cut into the five same tetrahedra respectively by the five half planes each of which includes one side edge and is perpendicular to the base. The volumes of the tetrahedra are calculated by using the formula in \cite{U}, for example.
\begin{table}[ht]
\centering
\begin{tabular}{|c | c|} 
\hline
$n$ & $\hspace{0.2cm}\pi/2n * \log(|\!\left<\!\left<\,\Gamma_{1, n} \,\right>\!\right>_{\rm CM}\!|)\hspace{0.2cm}$  \\ 
\hline
 24 & 3.440461579\\ 
\hline
 48 & 3.653711604\\ 
 \hline
 72 & 3.741389781\\ 
 \hline
 120 & 3.824412968\\ 
 \hline
 240 & 3.900858767\\ 
 \hline
 600 & 3.959111011\\ 
 \hline
 900 & 3.986845460\\ 
\hline
 1200 & 3.983212863\\ 
\hline
\hline
Vol. & 4.01536 \\
\hline
 \end{tabular}
 \hspace{0.1cm}
\begin{tabular}{|c | c|} 
\hline
$n$ & $\hspace{0.2cm}\pi/2n * \log(|\!\left<\!\left<\,\Gamma_{2, n}\,\right>\!\right>_{\rm CM}\!|)\hspace{0.2cm}$  \\ 
\hline
 24 & 2.597867632\\ 
\hline
 48 & 2.603012089\\ 
 \hline
 72 & 2.594717180\\ 
 \hline
 120 & 2.581960275\\ 
 \hline
 240 & 2.566522540\\ 
 \hline
 600 & 2.552634374\\ 
 \hline
 900 & 2.548604613\\ 
\hline
 1200 & 2.546357648\\ 
\hline
\hline
 Vol. & 2.53735\\
 \hline
 \end{tabular}\\
 \vspace{0.3cm}
 \begin{tabular}{|c | c|} 
\hline
$n$ & $\hspace{0.2cm}\pi/2n * \log(|\!\left<\!\left<\,\Gamma_{3, n}\, \right>\!\right>_{\rm CM}\!|)\hspace{0.2cm}$  \\ 
\hline
 12 & 3.4615752171\\ 
\hline
 24 & 3.6087612014\\ 
 \hline
 60 & 3.6418032698\\ 
 \hline
 90 & 3.6386525813\\ 
 \hline
 120 & 3.6346025832\\ 
 \hline
 240 & 3.6238281500\\ 
 \hline
 600 & 3.6126744280\\ 
\hline
\hline
Vol. & 3.59919 \\
\hline
 \end{tabular}
 \hspace{0.1cm}
\begin{tabular}{|c | c|} 
\hline
$n$ & $\hspace{0.2cm}\pi/2n * \log(|\!\left<\!\left<\,\Gamma_{4, n}\,\right>\!\right>_{\rm CM}\!|)\hspace{0.2cm}$  \\ 
\hline
 10 & 2.5883206638\\ 
\hline
 20 & 2.7020205533\\ 
 \hline
 60 & 2.6486057922\\ 
 \hline
 90 & 2.6180155598\\ 
 \hline
 120 & 2.5981240519\\ 
 \hline
 240 & 2.5593407688\\ 
 \hline
 600 & 2.5269587755\\ 
\hline
\hline
 Vol. & 2.49338\\
 \hline
 \end{tabular}
  \caption{Numerical calculations at $\varepsilon=10^{-30}$.} \label{table1}
 \end{table}
 The results seem to converge to the volumes of the pyramids. These near-integer colors calculations also show that the formulas seem to be well-defined at integer colors. The results are not so strong supporting evidences for Conjecture \ref{con2}. Finally, we propose a problem.
\begin{prob}
Prove Conjecture \ref{con2} for some polyhedra which have vertices whose valencies are more than three.
\end{prob}

\appendix

\section{Symmetry relations of $\{\,\cdot\,\}_{tet}$} \label{app01}
\par
From the isotopies of a tetrahedron graph, we have the symmetry relations of $\{\,\cdot\,\}_{tet}$ as follows:
\begin{eqnarray*}
& &\hspace{0.2cm} \left\{ \begin{array}{ccc}  a & b & c \\ d & e & f \\ \end{array} \right\}_{tet} 
= \left\{ \begin{array}{ccc}  c & d & e \\ \overline{f} & a & \overline{b} \\ \end{array} \right\}_{tet} 
=\left\{ \begin{array}{ccc}  e & \overline{f} & a \\ b & c & \overline{d} \\ \end{array} \right\}_{tet}\\
& &\hspace{-0.2cm}
= \left\{ \begin{array}{ccc}  \overline{f} & \overline{a} & \overline{e} \\ c & \overline{d} & b \\ \end{array} \right\}_{tet} \nonumber
= \left\{ \begin{array}{ccc}  \overline{e} & c & \overline{d} \\ \overline{b} & \overline{f} & a \\ \end{array} \right\}_{tet}
=\left\{ \begin{array}{ccc}  \overline{d} & \overline{b} & \overline{f} \\ \overline{a} & \overline{e} & \overline{c} \\ \end{array} \right\}_{tet}\\ 
& &\hspace{-0.2cm}= \left\{ \begin{array}{ccc}  f & \overline{d} & b \\ \overline{c} & \overline{a} & \overline{e} \\ \end{array} \right\}_{tet}
= \left\{ \begin{array}{ccc}  b & \overline{c} & \overline{a} \\ e & f & d \\ \end{array} \right\}_{tet}
= \left\{ \begin{array}{ccc}  \overline{a} & e & f \\ \overline{d} & b & c \\ \end{array} \right\}_{tet}\\ 
& &\hspace{-0.2cm}= \left\{ \begin{array}{ccc}  \overline{c} & a & \overline{b} \\ f & d & e \\ \end{array} \right\}_{tet}
= \left\{ \begin{array}{ccc}  \overline{b} & f & d \\ \overline{e} & \overline{c} & \overline{a} \\ \end{array} \right\}_{tet} 
= \left\{ \begin{array}{ccc}  d & \overline{e} & \overline{c} \\ a & \overline{b} & \overline{f} \\ \end{array} \right\}_{tet}\\ 
& &\hspace{-0.2cm}= \left\{ \begin{array}{ccc}  \overline{c} & b & \overline{a} \\ \overline{f} & \overline{e} & \overline{d} \\ \end{array} \right\}_{tet}
= \left\{ \begin{array}{ccc}  \overline{e} & d & \overline{c} \\ b & \overline{a} & f \\ \end{array} \right\}_{tet}
= \left\{ \begin{array}{ccc}  \overline{a} & \overline{f} & \overline{e} \\ d & \overline{c} & \overline{b} \\ \end{array} \right\}_{tet}\\ 
& &\hspace{-0.2cm}= \left\{ \begin{array}{ccc}  e & \overline{a} & f \\ \overline{b} & d & \overline{c} \\ \end{array} \right\}_{tet} 
= \left\{ \begin{array}{ccc}  d & c & e \\ \overline{a} & f & b \\ \end{array} \right\}_{tet}
= \left\{ \begin{array}{ccc}  f & \overline{b} & d \\ c & e & a \\ \end{array} \right\}_{tet}\\
& &\hspace{-0.2cm}= \left\{ \begin{array}{ccc}  \overline{b} & \overline{d} & \overline{f} \\ e & a & c \\ \end{array} \right\}_{tet}
= \left\{ \begin{array}{ccc}  a & \overline{c} & \overline{b} \\ \overline{d} & \overline{f} & \overline{e} \\ \end{array} \right\}_{tet}
= \left\{ \begin{array}{ccc}  \overline{f} & e & a \\ \overline{c} & \overline{b} & d \\ \end{array} \right\}_{tet}\\ 
& &\hspace{-0.2cm}= \left\{ \begin{array}{ccc}  b & a & c \\ \overline{e} & \overline{d} & \overline{f} \\ \end{array} \right\}_{tet}
= \left\{ \begin{array}{ccc}  \overline{d} & f & b \\ a & c & e \\ \end{array} \right\}_{tet}
= \left\{ \begin{array}{ccc}  c & \overline{e} & \overline{d} \\ f & b & \overline{a} \\ \end{array} \right\}_{tet}.
\end{eqnarray*}
Here, the first four lines come from the counterclockwise $2\pi/3$ rotations of the following four diagrams respectively:
$$ 
\raisebox{-21 pt}{\begin{overpic}[bb=0 0 122 140, width=52 pt]{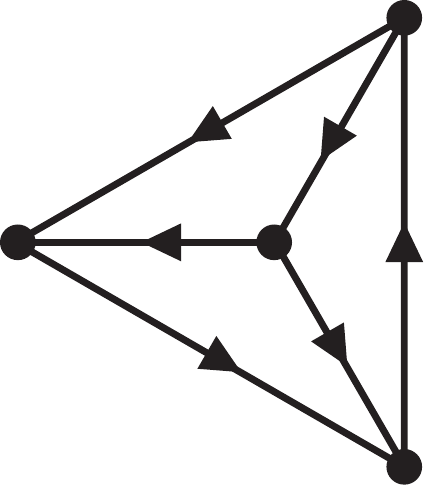}
\put(20,47){$a$}\put(23, 21){$b$}\put(20, 8){$c$}\put(40,21){$d$}
\put(53, 27){$e$}\put(33, 40){$f$}
\end{overpic}}
\hspace{0.3cm}
,
\hspace{0.5cm}
\raisebox{-21 pt}{\begin{overpic}[bb=0 0 122 140, width=52 pt]{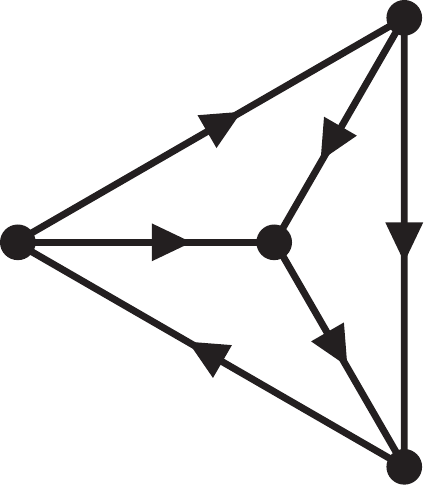}
\put(21,23){$a$}\put(34,40){$b$}\put(41,21){$c$}\put(53,28){$d$}
\put(20,8){$e$}\put(20,47){$f$}
\end{overpic}}
\hspace{0.3cm}
,
\hspace{0.5cm}
\raisebox{-21 pt}{\begin{overpic}[bb=0 0 122 140, width=52 pt]{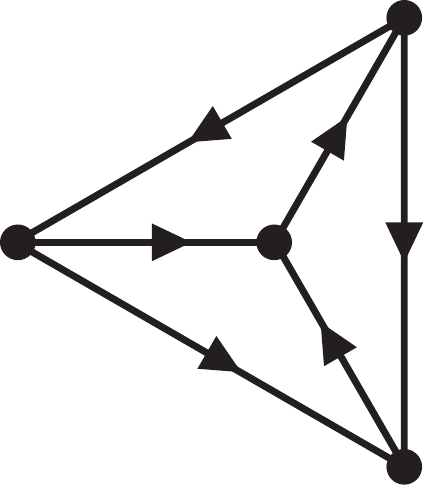}
\put(53,29){$a$}\put(20,7){$b$}\put(42,19){$c$}\put(21,21){$d$}
\put(33,40){$e$}\put(20,47){$f$}
\end{overpic}}
\hspace{0.3cm}
,
\hspace{0.5cm}
\raisebox{-21 pt}{\begin{overpic}[bb=0 0 122 140, width=52 pt]{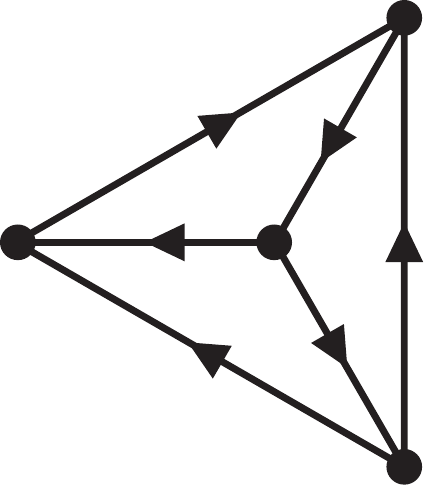}
\put(22,22){$a$}\put(20,7){$b$}\put(20,47){$c$}\put(53,27){$d$}
\put(33,40){$e$}\put(40,23){$f$}
\end{overpic}}
\hspace{0.3cm}.
$$
The second half part comes from the $\pi$ rotation of the first half part with respect to the horizontal line in the paper.

\section{Proof of Lemma \ref{prop01}} \label{app02}
\par
According to \cite{CM}, we say that a triple of integers $(i,j,k)$ is \textit{admissible} if it satisfies the following conditions: 
$$0<i,\, j,\, k<n-1, \hspace{0.5cm} n-1<i+j+k<2(n-1)$$
and
$$0<i+j-k,\,\, j+k-i,\,\, k+i-j<n-1.$$
Let $\Gamma$ be a framed oriented spatial trivalent graph with integer colors. $\Gamma$ is called \textit{admissible} if every triple of colors around the vertices is admissible. If $(i, j, k)$ is admissible, $(\overline{i},j,k)$, $(i, \overline{j},k)$ and $(i,j,\overline{k})$ are also admissible. This induces that if $\Gamma$ is admissible, a graph obtained by reversing the orientation of an edge of $\Gamma$ and having the same colors as $\Gamma$ is also admissible.
Thus the admissible condition of integer colored spatial graphs does not depend on the edge orientations. 
The formula of $\{\,\cdot\,\}_{tet}$ is deformed for integer colors as follows:
\begin{eqnarray*}
& &\hspace{-0.9cm}\left\{ \begin{array}{ccc}  a & b & c \\ d & e & f \\ \end{array} \right\}_{tet} = (-1)^{n-1}\frac{\{B_{cde}\}!\{B_{abc}\}!}{\{B_{bdf}\}!\{B_{afe}\}!}\left[ \begin{array}{c} 2c\\ A_{abc}+1-n \end{array} \right]\left[ \begin{array}{c} 2c\\ B_{ced} \end{array} \right]^{-1}\\
& &\hspace{2.2cm}\times \sum_{z=s}^{S}\left[ \begin{array}{c} A_{afe}+1-n \\ 2e+z+1-n \end{array} \right]\left[ \begin{array}{c} B_{aef}+z\\ B_{aef} \end{array} \right]\left[ \begin{array}{c} B_{bfd}+B_{cde}-z\\ B_{bfd} \end{array} \right]\left[ \begin{array}{c} B_{dec}+z\\ B_{dfb} \end{array} \right],
\end{eqnarray*}
where $s=\max(0, n-1-2e, B_{dfb}-B_{dec}, B_{bfd}+B_{cde}-(n-1))$, $S=\min(B_{cde}, B_{afe}, n-1-B_{aef}, n-1-B_{dec})$. For $s\leq z\leq S$, the binomials in the summation are non-zero. On the other hand, by using the relation (\ref{sym}), 
\begin{eqnarray*}
& &\hspace{-0.9cm}\left\{ \begin{array}{ccc} \overline{a} & \overline{b} & \overline{c} \\ \overline{d} & \overline{e} & \overline{f} \\ \end{array} \right\}_{tet}=\left\{ \begin{array}{ccc}  d & b & f \\ a & e & c \\ \end{array} \right\}_{tet}\\
& &\hspace{1.8cm} = (-1)^{n-1}\frac{\{B_{afe}\}!\{B_{bdf}\}!}{\{B_{abc}\}!\{B_{cde}\}!}\left[ \begin{array}{c} 2f\\ A_{bdf}+1-n \end{array} \right]\left[ \begin{array}{c} 2f\\ B_{efa} \end{array} \right]^{-1}\\
& &\hspace{2.2cm}\times \sum_{z=s'}^{S'}\left[ \begin{array}{c} A_{cde}+1-n \\ 2e+z+1-n \end{array} \right]\left[ \begin{array}{c} B_{dec}+z\\ B_{dec} \end{array} \right]\left[ \begin{array}{c} B_{bca}+B_{afe}-z\\ B_{bca} \end{array} \right]\left[ \begin{array}{c} B_{aef}+z\\ B_{acb} \end{array} \right],
\end{eqnarray*}
where $s'=\max(0, n-1-2e, B_{acb}-B_{aef}, B_{bca}+B_{afe}-(n-1))$, $S'=\min(B_{afe}, B_{cde}, n-1-B_{dec}, n-1-B_{aef})$. Notice that $s=s'$, $S=S'$. Thus,
\begin{eqnarray*}
& &\hspace{-1.0cm}\left\{ \begin{array}{ccc}  a & b & c \\ d & e & f \\ \end{array} \right\}_{tet} \left\{ \begin{array}{ccc} \overline{a} & \overline{b} & \overline{c} \\ \overline{d} & \overline{e} & \overline{f} \\ \end{array} \right\}_{tet} =\\
& &\hspace{2.4cm}\left[ \begin{array}{c} 2c\\ A_{abc}+1-n \end{array} \right]\left[ \begin{array}{c} 2c\\ B_{ced} \end{array} \right]^{-1} \left[ \begin{array}{c} 2f\\ A_{bdf}+1-n \end{array} \right]\left[ \begin{array}{c} 2f\\ B_{efa} \end{array} \right]^{-1}\\
& &\hspace{2.5cm}\times \sum_{z=s}^{S}\left[ \begin{array}{c} A_{afe}+1-n \\ 2e+z+1-n \end{array} \right]\left[ \begin{array}{c} B_{aef}+z\\ B_{aef} \end{array} \right]\left[ \begin{array}{c} B_{bfd}+B_{cde}-z\\ B_{bfd} \end{array} \right]\left[ \begin{array}{c} B_{dec}+z\\ B_{dfb} \end{array} \right]\\
& &\hspace{2.5cm}\times \sum_{z=s}^{S}\left[ \begin{array}{c} A_{cde}+1-n \\ 2e+z+1-n \end{array} \right]\left[ \begin{array}{c} B_{dec}+z\\ B_{dec} \end{array} \right]\left[ \begin{array}{c} B_{bca}+B_{afe}-z\\ B_{bca} \end{array} \right]\left[ \begin{array}{c} B_{aef}+z\\ B_{acb} \end{array} \right].
\end{eqnarray*}
Remark that $s=\max (0, n-1-2e, c+f-b-e, c+f-\overline{b}-e)$. Thus $s$ and $S$ are invariant under the change $b\rightarrow \overline{b}$. We show 
$$
\left\{ \begin{array}{ccc}  a & b & c \\ d & e & f \\ \end{array} \right\}_{tet} \left\{ \begin{array}{ccc} \overline{a} & \overline{b} & \overline{c} \\ \overline{d} & \overline{e} & \overline{f} \\ \end{array} \right\}_{tet} 
=
\left\{ \begin{array}{ccc}  a & \overline{b} & c \\ d & e & f \\ \end{array} \right\}_{tet} \left\{ \begin{array}{ccc} \overline{a} & b & \overline{c} \\ \overline{d} & \overline{e} & \overline{f} \\ \end{array} \right\}_{tet}.
$$
We focus on the parts depending on the variable $b$ and prove 
\begin{eqnarray} \label{eq10}
& &\hspace{0.0cm}\left[ \begin{array}{c} 2f\\ A_{bdf}+1-n \end{array} \right]\left[ \begin{array}{c} B_{bfd}+B_{cde}-z\\ B_{bfd} \end{array} \right]\left[ \begin{array}{c} B_{dec}+z\\ B_{dfb} \end{array} \right]\\
 \nonumber\\
& &\hspace{6.0cm}=\left[ \begin{array}{c} 2f\\ A_{\overline{b}df}+1-n \end{array} \right]\left[ \begin{array}{c} B_{\overline{b}fd}+B_{cde}-z\\ B_{\overline{b}fd} \end{array} \right]\left[ \begin{array}{c} B_{dec}+z\\ B_{df\overline{b}} \end{array} \right] \nonumber
\end{eqnarray}
and
\begin{eqnarray} \label{eq11}
& &\hspace{0.0cm}\left[ \begin{array}{c} 2c\\ A_{abc}+1-n \end{array} \right] \left[ \begin{array}{c} B_{bca}+B_{afe}-z\\ B_{bca} \end{array} \right]\left[ \begin{array}{c} B_{aef}+z\\ B_{acb} \end{array} \right]\\
 \nonumber\\
& &\hspace{6.0cm}=\left[ \begin{array}{c} 2c\\ A_{a\overline{b}c}+1-n \end{array} \right] \left[ \begin{array}{c} B_{\overline{b}ca}+B_{afe}-z\\ B_{\overline{b}ca} \end{array} \right]\left[ \begin{array}{c} B_{aef}+z\\ B_{ac\overline{b}} \end{array} \right] \nonumber
\end{eqnarray}
for each $s\leq z\leq S$. We deform (\ref{eq10}) to 
\begin{eqnarray} \label{eq12}
& &\hspace{0.1cm}\left[ \begin{array}{c} 2f\\ A_{\overline{b}df}+1-n \end{array} \right]^{-1}\left[ \begin{array}{c} B_{bfd}+B_{cde}-z\\ B_{bfd} \end{array} \right]\left[ \begin{array}{c} B_{dec}+z\\ B_{dfb} \end{array} \right]\\
 \nonumber\\
& &\hspace{5.7cm}=\left[ \begin{array}{c} 2f\\ A_{bdf}+1-n \end{array} \right]^{-1}\left[ \begin{array}{c} B_{\overline{b}fd}+B_{cde}-z\\ B_{\overline{b}fd} \end{array} \right]\left[ \begin{array}{c} B_{dec}+z\\ B_{df\overline{b}} \end{array} \right]. \nonumber
\end{eqnarray}
The left hand side is equal to 
\begin{eqnarray}
\nonumber & &\hspace{0.0cm}\frac{\{b+f-d\}!}{\displaystyle\prod_{j=0}^{b+f-d-1}\{2f-j\}}\frac{\displaystyle \prod_{j=0}^{c+d-e-z-1}\{b+c+f-e-z-j\}}{\{c+d-e-z\}!}\frac{\displaystyle \prod_{j=0}^{b+e-c-f+z-1}\{d+e-c+z-j\}}{\{b+e-c-f+z\}!}\\
& &\hspace{1.4cm}=\frac{\{b+c+f-e-z\}!}{\{b+e-c-f+z\}!}\times\frac{1}{\{c+d-e-z\}!}\frac{\displaystyle \prod_{j=0}^{b+e-c-f+z-1}\{d+e-c+z-j\}}{\displaystyle\prod_{j=0}^{b+f-d-1}\{2f-j\}}. \label{eq13}
\end{eqnarray}
The latter part of (\ref{eq13}) is equal to
\begin{eqnarray*}
& &\hspace{-0.4cm}\frac{1}{\{c+d-e-z\}!}\frac{\displaystyle \prod_{j=1}^{b+e-c-f+z}\{d+f-b+j\}}{\displaystyle\prod_{j=1}^{b+f-d}\{d+f-b+j\}}\\
& &\hspace{0.3cm}=\left\{ \begin{array}{l c} \displaystyle \frac{1}{\{c+d-e-z\}!} \prod_{j=1}^{d+e-c+z-2f}\{2f+j\} & \hspace{0.1cm}{\rm if }\hspace{0.2cm} b+f-d \leq b+e-c-f+z, \\
\\
 \displaystyle \frac{\displaystyle 1}{\displaystyle \{c+d-e-z\}!} \frac{\displaystyle 1}{\displaystyle \prod_{j=1}^{2f-d-e+c-z} \{d+e-c+z+j\}} & \hspace{0.1cm}{\rm if }\hspace{0.2cm} b+f-d > b+e-c-f+z, \end{array}
\right.
\end{eqnarray*}
and it does not depend on $b$. Since $b+c+f-e-z$ and $b+e-c-f+z$ are in $\{0, 1, \dots, n-1\}$, using (\ref{eq00}), the former part of (\ref{eq13}) is
$$\frac{\{b+c+f-e-z\}!}{\{b+e-c-f+z\}!}=\frac{\,\,\displaystyle\frac{\{n-1\}!}{\{\overline{b}-c-f+e+z\}!}\,\,}{\displaystyle\frac{\{n-1\}!}{\{\overline{b}-e+c+f-z\}!}}=\frac{\{\overline{b}+c+f-e-z\}!}{\{\overline{b}+e-c-f+z\}!}.$$
This is equal to the counterpart of the right hand side of (\ref{eq12}), and (\ref{eq10}) holds. The relation (\ref{eq11}) is proved similarly. This completes the proof of Lemma \ref{prop01}.

\end{document}